\documentclass{article}
\usepackage[preprint,nonatbib]{neurips_2026}

\usepackage[utf8]{inputenc}
\usepackage[T1]{fontenc}
\usepackage{amsmath,amsthm,amssymb,enumerate,subfigure,multirow,hhline,color}
\usepackage{xcolor}
\usepackage{hyperref}
\usepackage{url}
\usepackage{microtype}
\usepackage{graphicx}
\usepackage{wrapfig}
\usepackage{graphics}
\usepackage{tikz}
\usetikzlibrary{arrows.meta,calc,decorations.markings}
\usepackage[sort]{cite}
\usepackage{algorithm}
\usepackage{algpseudocode}
\usepackage{epsfig,psfrag}
\usepackage{tabularx}
\usepackage{tabulary}
\usepackage[noabbrev,capitalise,nameinlink]{cleveref}

\hypersetup{
breaklinks=true,
colorlinks=true,
linkcolor=blue,
citecolor=blue,
urlcolor=blue
}

\DeclareMathOperator{\argmax}{arg\,max}
\DeclareMathOperator{\Argmax}{Arg\,max}
\DeclareMathOperator{\dist}{dist}

\newcommand{\proj}{\mathbf{\Pi}_\perp}

\newtheorem{theorem}{Theorem}

\newtheorem{assumption}{Assumption}

\newtheorem{corollary}{Corollary}

\newtheorem{lemma}{Lemma}
\newtheorem{definition}{Definition}
\newtheorem{proposition}{Proposition}

\allowdisplaybreaks

\title{Beyond the Bellman Fixed Point: Geometry and Fast Policy Identification in Value Iteration}

\author{Donghwan Lee\\
Department of Electrical Engineering, Korea Advanced Institute of Science and Technology (KAIST),\\
Daejeon 34141, South Korea\\
{\tt donghwan@kaist.ac.kr}
}

\begin{document}

\maketitle

\begin{abstract}
Q-value iteration (Q-VI) is usually analyzed through the \(\gamma\)-contraction of the Bellman operator. This argument proves convergence to \(Q^*\), but it gives only a coarse account of when the induced greedy policy becomes optimal. We study discounted Q-VI as a switching system and focus on the practically optimal solution set (POSS), the set of \(Q\)-functions whose tie-broken greedy policies are optimal. The main result shows that Q-VI reaches the optimal action class in finite time by entering an invariant tube around \(\mathcal X_1=Q^*+\operatorname{span}(\mathbf 1)\), which is contained in the POSS. For every \(\varepsilon>0\), the distance to \(\mathcal X_1\) satisfies an exponential bound with rate \((\bar\rho+\varepsilon)^k\), where \(\bar\rho\) is the joint spectral radius of the projected switching family restricted to directions transverse to \(\mathcal X_1\). When \(\bar\rho<\gamma\), this transverse convergence is faster than the classical contraction rate. The analysis separates fast policy identification from the subsequent convergence to \(Q^*\), which may still be governed by the all-ones mode. We also give spectral and graph-theoretic conditions under which the strict inequality \(\bar\rho<\gamma\) holds or fails.
\end{abstract}

\section{Introduction}

Dynamic programming~\cite{bellman1966dynamic,bertsekas1996neuro,bertsekas2015dynamic} is a standard method for solving Markov decision processes~\cite{puterman2014markov}. Among its basic algorithms, Q-value iteration (Q-VI) is notable for its simple update rule and its classical contraction guarantee~\cite{bertsekas1996neuro,bertsekas2015dynamic}. In a discounted problem with discount factor $\gamma \in (0,1)$, the Bellman optimality operator is a \(\gamma\)-contraction in the infinity norm. Hence the Q-VI iterates converge exponentially to the optimal Q-function \(Q^*\) at rate $\gamma$, a result that underlies much of dynamic programming and reinforcement learning~\cite{sutton1998reinforcement}.

\begin{wrapfigure}{r}{0.44\textwidth}
\vspace{-10pt}
\centering
\begin{tikzpicture}[
    scale=0.70,
    transform shape,
    line cap=round,
    line join=round,
    axis/.style={-{Stealth[length=2.4mm]}, line width=0.65pt, black!65},
    xoneline/.style={line width=1.15pt, blue!70!black, dash pattern=on 4pt off 2pt},
    tubeedge/.style={line width=0.65pt, blue!35},
    traj/.style={
        line width=1.35pt,
        red!80!black,
        dash pattern=on 1.0pt off 2.1pt,
        postaction={decorate},
        decoration={markings,
            mark=at position 0.24 with {
                \arrow[solid, red!80!black]{Stealth[length=3.8mm,width=3.0mm]}
            },
            mark=at position 0.63 with {
                \arrow[solid, red!80!black]{Stealth[length=3.8mm,width=3.0mm]}
            }
        }
    }
]

\node[font=\large\bfseries] at (2.65,5.45) {$Q$ parameter space};

\filldraw[fill=black!4, draw=black!35, line width=0.7pt]
plot[smooth cycle, tension=0.55] coordinates {
    (-1.55,-1.75) (-1.55,-0.35) (-0.70,0.75) (0.25,1.65)
    (1.20,2.55) (2.25,3.55) (3.35,4.35) (4.35,4.95)
    (4.75,4.70) (4.85,4.10) (4.55,3.45) (3.70,2.55)
    (2.70,1.55) (1.65,0.55) (0.65,-0.40) (-0.25,-1.15)
    (-0.95,-1.65)
};

\fill[blue!16]
    (-1.10,-1.40) -- (4.25,3.95) -- (4.25,4.55) -- (-1.10,-0.80) -- cycle;
\draw[tubeedge] (-1.10,-1.40) -- (4.25,3.95);
\draw[tubeedge] (-1.10,-0.80) -- (4.25,4.55);

\draw[xoneline] (-1.10,-1.10) -- (4.25,4.25);

\draw[axis] (-1.65,0) -- (6.95,0);
\draw[axis] (0,-1.30) -- (0,4.45);

\filldraw[fill=blue!70!black, draw=blue!90!black, line width=0.55pt]
    (0,0) circle (2.7pt);

\node[anchor=south east] at (-0.10,0.15) {$Q^{\ast}$};

\draw[traj]
    (6.05,1.18)
    .. controls (5.60,1.28) and (5.15,1.56) .. (4.75,1.70)
    .. controls (4.15,1.94) and (3.65,2.18) .. (3.05,2.45)
    .. controls (2.75,2.58) and (2.55,2.40) .. (2.30,2.05)
    .. controls (1.92,1.55) and (1.35,1.20) .. (0.78,0.74)
    .. controls (0.42,0.40) and (0.16,0.15) .. (0.02,0.02);

\filldraw[black] (6.05,1.18) circle (1.55pt);
\node[right=2pt] at (6.05,1.18) {$Q_0$};

\filldraw[black!70] (2.30,2.05) circle (1.45pt);
\node[above left=1pt] at (2.30,2.05) {$Q_{K_{\rm id}}$};

\draw[<->, black!55, line width=0.55pt]
    (3.20,3.50) -- (3.50,3.20)
    node[midway, right=3pt] {$\mathcal{T}_{\delta}$};

\node[rotate=45, anchor=south, blue!70!black] at (3.05,3.20)
    {$\mathcal{X}_1=Q^{\ast}+\mathrm{span}(\mathbf 1)$};

\node[black!65, anchor=west] at (5.02,4.55) {$\mathcal{X}^{\ast}$};

\node[red!80!black] at (4.75,1.35) {$Q_k$};

\end{tikzpicture}
\caption{Geometric mechanism for finite-time POSS identification.}
\label{fig:overall-geometric-picture}
\vspace{-12pt}
\end{wrapfigure}
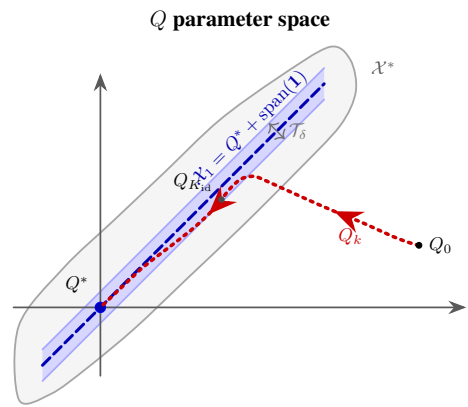

The geometric mechanism considered in this paper is shown schematically in~\cref{fig:overall-geometric-picture}. The usual contraction proof controls the distance to the fixed point \(Q^*\), but it does not resolve the geometry of the trajectory in directions that do not affect action preferences. This distinction matters when the objective is to determine when the greedy policy has already become optimal, even though the Q-function itself has not yet converged. To make this question precise, we define the practically optimal solution set (POSS), denoted by \(\mathcal X^*\), as the set of \(Q\)-functions whose tie-broken greedy policies are optimal.

The starting point of the paper is a switching-system representation of Q-VI. The Bellman update can be written as an affine switching system~\cite{liberzon2003switching,lin2009stability,hu2017resilient}, with the active mode selected by the tie-broken greedy policy of the current iterate. This representation connects Q-VI with tools from switching-system theory~\cite{liberzon2003switching,lin2009stability,hu2017resilient} and Lyapunov analysis~\cite{khalil2002nonlinear}. It also singles out the affine space
\[
\mathcal X_1 := Q^* + \operatorname{span}({\bf 1}).
\]
A shift of \(Q^*\) in the all-ones direction preserves every statewise action ordering, so \(\mathcal X_1\subset\mathcal X^*\). We prove that a sufficiently small tube around \(\mathcal X_1\) is not only contained in \(\mathcal X^*\), but also invariant under the Bellman operator. Once the iterates enter this tube, the tie-broken greedy policy is optimal, and the entrance time is finite.

The rate of this entrance is governed by the dynamics transverse to \(\mathcal X_1\). The classical contraction argument treats the all-ones direction and the transverse directions together, giving the single rate \(\gamma\). By deriving an exact stochastic-policy switching representation for the Q-VI error and projecting out the all-ones direction, we obtain a sharper bound: the distance from \(Q_k\) to \(\mathcal X_1\) decays exponentially at any rate larger than the joint spectral radius (JSR)~\cite{tsitsiklis1997lyapunov,rota1960note,blondel2005computationally} of the restricted projected switching family. If this restricted JSR is smaller than \(\gamma\), then Q-VI approaches \(\mathcal X_1\), and therefore identifies the POSS, faster than the standard fixed-point contraction rate would suggest. The resulting trajectory has two stages: rapid transverse approach and policy identification, followed by convergence to \(Q^*\), possibly dominated by the remaining all-ones component.

We also study when the restricted JSR is strictly below the classical rate. The conditions are expressed in terms of uniform scrambling products, Dobrushin coefficients, and span-seminorm contractions~\cite{dobrushin1956central,hajnal1958weak,wolfowitz1963products,seneta2006non}. We also record the main Markov-chain obstructions: if a deterministic policy induces a nontrivial unit-modulus mode, for example through multiple closed communicating classes or through periodicity of a closed class, then the full restricted family cannot have JSR smaller than \(\gamma\). These conditions identify the mixing structure needed for the transverse dynamics to improve on the classical contraction bound.

\section{Related work}

Switching-system methods have recently been used to analyze reinforcement learning~\cite{sutton1998reinforcement} and dynamic programming~\cite{bellman1966dynamic,bertsekas1996neuro,bertsekas2015dynamic}. The framework of~\cite{lee2020unified} gives a unified switching-system view of Q-learning algorithms, and~\cite{lee2023discrete} develops a discrete-time formulation with convergence guarantees. Subsequent work established sharper final-iteration bounds~\cite{lee2024final} and finite-time results under Markovian observations and diminishing step sizes~\cite{lim2024finite}.

For Q-VI, several papers have examined geometric or Lyapunov-based structure. A switching-system interpretation and orthant-wise geometric behavior are studied in~\cite{lee2023some}. Convex-analytic and Lyapunov methods for value-based reinforcement learning appear in~\cite{guo2022convex}, while~\cite{iervolino2023lyapunov} formulates value iteration as an affine switching system. Switching-based policy-selection mechanisms are considered in~\cite{tipaldi2025switching}. Goyal and Grand-Cl\'ement~\cite{goyal2023first} develop a first-order approach to accelerated value iteration and analyze value-iteration-type recursions through compositions of affine maps, or equivalently through a linear time-varying/switching dynamical-system viewpoint. These works are close in spirit to the present paper, as they use dynamical-systems tools to study value-based dynamic programming and reinforcement learning algorithms.

Geometric approaches outside switching systems provide a complementary perspective. The value-function polytope framework of~\cite{dadashi2019value} characterizes geometric and topological properties of value functions in finite MDPs and uses that structure to interpret reinforcement learning algorithms. Related polyhedral ideas are used in~\cite{wu2022geometric}, which develops geometric policy iteration through hyperplane arrangements and boundary structures of MDPs. Those works focus mainly on value-function geometry, policy improvement, and polyhedral structure. Our focus is instead on the Q-VI trajectory relative to \(\mathcal X_1=Q^*+\operatorname{span}(\mathbf 1)\) and the POSS \(\mathcal X^*\).

Finite-time policy identification is also related to turnpike results for discounted finite MDPs. Feinberg and He~\cite{feinberg2025properties} study turnpike functions and turnpike integers, showing that, after a finite number of value-iteration steps, the generated decision rules are deterministic optimal policies for the infinite-horizon problem. This is close in spirit to finite identification, but their analysis is formulated through turnpike functions for VI, whereas our results study Q-VI trajectory geometry, invariant tubes around \(\mathcal X_1\), and projected switching dynamics.

Other analyses of Q-VI rely on probabilistic or nonsmooth viewpoints. The works~\cite{mustafin2024value,mustafin2025analysis,mustafin2025geometric} use tools such as absolute probability sequences to refine the convergence analysis of classical MDP solution methods. Semismooth Newton-type interpretations of dynamic programming are studied in~\cite{gargiani2022dynamic}. These nonsmooth-equation perspectives are relevant because Bellman optimality operators contain maximization and policy selection. Here we use the same nonsmooth structure to obtain switching dynamics and finite-time identification of the optimal action class.

The technical tools are also connected to stability theory for switching systems and to JSR analysis. The JSR was introduced in~\cite{rota1960note}, and Lyapunov methods for bounding or approximating it have been extensively developed; see, for example,~\cite{tsitsiklis1997lyapunov}. Our Lyapunov construction follows this general line but is applied after projecting away the all-ones direction. The projection is also related to seminorm-based analyses in reinforcement learning and dynamic programming. In particular, the non-asymptotic seminorm Lyapunov framework of~\cite{chen2025non} studies deterministic and stochastic iterative algorithms through seminorm contractions.

Most existing results emphasize asymptotic convergence, global contraction, or the geometry of value functions and policies. The present paper instead analyzes the geometry of the Q-VI trajectory itself, with an emphasis on finite-time policy identification, invariant tubes, and projected switching dynamics.

\section{Preliminaries}\label{sec:preliminaries}

\subsection{Notations}
We use the following notation. The symbols \({\mathbb R}\), \({\mathbb R}^n\), and \({\mathbb R}^{n\times m}\) denote the set of real numbers, the \(n\)-dimensional Euclidean space, and the set of \(n\times m\) real matrices, respectively. For a matrix \(A\), \(A^\top\) denotes its transpose. The relations \(A\succ0\), \(A\prec0\), \(A\succeq0\), and \(A\preceq0\) mean that \(A\) is symmetric positive definite, negative definite, positive semidefinite, and negative semidefinite, respectively. The identity matrix with appropriate dimensions is denoted by \(I\). For a finite set \(\mathcal S\), its cardinality is denoted by \(|\mathcal S|\). The Kronecker product of \(A\) and \(B\) is denoted by \(A\otimes B\). The vector with all entries equal to one is denoted by \(\mathbf 1\), and \(\lambda_{\max}(\cdot)\) and \(\lambda_{\min}(\cdot)\) denote the maximum and minimum eigenvalues, respectively.

For a convex set \(\mathcal Y \subset \mathbb R^n\) and a vector \(x\in\mathbb R^n\), \(\dist_2(x,\mathcal Y)\) and \(\dist_\infty(x,\mathcal Y)\) denote the Euclidean distance and the infinity-norm distance from \(x\) to \(\mathcal Y\), respectively:
\[
\dist_2(x,\mathcal Y):=\inf_{y\in\mathcal Y}\|x-y\|_2,
\qquad
\dist_\infty(x,\mathcal Y):=\inf_{y\in\mathcal Y}\|x-y\|_\infty.
\]
We write \(\Delta_{|{\cal A}|}\) for the probability simplex over a finite action set \({\cal A}\):
\[
\Delta_{|{\cal A}|}
:=
\left\{
p\in\mathbb R^{|{\cal A}|}:\ p_i\ge 0,\ \sum_{i=1}^{{|{\cal A}|}}p_i=1
\right\}.
\]
Throughout the paper, \(\Argmax\) denotes the set-valued maximizer, while \(\argmax\) denotes a fixed tie-broken single-valued maximizer.

\subsection{Markov decision process}
We consider an infinite-horizon discounted Markov decision process (MDP)~\cite{puterman2014markov}, in which an agent sequentially chooses actions to maximize cumulative discounted rewards. The state and action spaces are finite and are denoted by \({\cal S}:=\{1,2,\ldots,|{\cal S}|\}\) and \({\cal A}:=\{1,2,\ldots,|{\cal A}|\}\), respectively. At state \(s\in{\cal S}\), the decision maker selects an action \(a\in{\cal A}\). The next state \(s'\) is drawn according to \(P(s'|s,a)\), and the transition incurs reward \(r(s,a,s')\), where \(r:{\cal S}\times{\cal A}\times{\cal S}\to{\mathbb R}\). We write \(r(s_k,a_k,s_{k+1})=:r_{k+1}\) for \(k\ge0\). A deterministic policy \(\pi:{\cal S}\to{\cal A}\) maps each state \(s\) to an action \(\pi(s)\). Throughout the paper, the discount factor satisfies \(\gamma\in(0,1)\).

For a policy \(\pi\), the Q-function under \(\pi\) is defined as
\[
Q^{\pi}(s,a)={\mathbb E}\left[ \left. \sum_{k=0}^\infty {\gamma^k r_{k+1}} \right|s_0=s,a_0=a,\pi \right],
\]
for all \(s\in{\cal S}\) and \(a\in{\cal A}\). The optimal Q-function is
\[
Q^*(s,a):=\sup_{\pi\in\Theta} Q^\pi(s,a),\qquad s\in{\cal S},\ a\in{\cal A},
\]
where \(\Theta\) denotes the set of all admissible deterministic policies. A deterministic policy \(\pi^*\) is optimal if \(Q^{\pi^*}(s,a)=Q^*(s,a)\) for all \((s,a)\in{\cal S}\times{\cal A}\). Once \(Q^*\) is known, an optimal tie-broken greedy policy can be recovered as \(\pi^*(s)=\argmax_{a\in {\cal A}}Q^*(s,a)\).

For each state \(s\in{\cal S}\), define the set of optimal greedy actions by
\[
\Phi^*(s):=\Argmax_{a\in {\cal A}}Q^*(s,a).
\]
The set of all optimal deterministic policies is then
\[
\Theta^*
:=
\{\pi\in\Theta:\ \pi(s)\in\Phi^*(s),\ \forall s\in {\cal S}\}.
\]
For completeness, we recall why this set coincides with the set of optimal deterministic policies in a finite discounted MDP. If \(\pi\in\Theta^*\), then
\[
V^*(s)
=
Q^*(s,\pi(s))
=
R(s,\pi(s))+
\gamma\sum_{s'\in {\cal S}}P(s'|s,\pi(s))V^*(s'),
\qquad \forall s\in {\cal S}.
\]
Therefore, \(V^*\) is a fixed point of the discounted Bellman evaluation operator for \(\pi\). By uniqueness of this fixed point, \(V^\pi=V^*\). Consequently,
\[
Q^\pi(s,a)
=
R(s,a)+\gamma\sum_{s'\in {\cal S}}P(s'|s,a)V^\pi(s')
=
R(s,a)+\gamma\sum_{s'\in {\cal S}}P(s'|s,a)V^*(s')
=
Q^*(s,a),
\]
and therefore, \(\pi\) is optimal. Conversely, if a deterministic policy \(\pi\) is optimal, then \(V^\pi=V^*\). Hence
\begin{align*}
V^*(s)
=&
V^\pi(s)
=
R(s,\pi(s))+\gamma\sum_{s'\in{\cal S}}P(s'|s,\pi(s))V^\pi(s')\\
=&
R(s,\pi(s))+\gamma\sum_{s'\in{\cal S}}P(s'|s,\pi(s))V^*(s')
=
Q^*(s,\pi(s)),
\end{align*}
and hence, \(\pi(s)\in\Phi^*(s)\) for all states. Therefore, the display above indeed describes all optimal deterministic policies.
The corresponding optimal value function is
\[
V^*(s):=\max_{a\in {\cal A}}Q^*(s,a).
\]

\subsection{Q-value iteration}
We study Q-value iteration (Q-VI)~\cite{bertsekas1996neuro}. Let \(F\) denote the Bellman operator
\[
(FQ)(s,a):= R(s,a) + \gamma \sum\limits_{s' \in {\cal S}} {P(s'|s,a)\max _{a' \in {\cal A}}Q(s',a')},\quad (s,a)\in {\cal S}\times {\cal A}.
\]
The Q-VI recursion is
\begin{equation}\label{eq:qvi-basic}
Q_{k+1}=FQ_k:=F(Q_k),\qquad k=0,1,\ldots .
\end{equation}
The corresponding pseudocode is given in Appendix~\ref{app:qvi-algorithm} as Algorithm~\ref{alg:qvi}. It is well known that Q-VI converges exponentially to \(Q^*\) in the infinity norm \(\|\cdot\|_\infty\)~\cite[Lemma~2.5]{bertsekas1996neuro}.
\begin{lemma}[\cite{bertsekas1996neuro}]\label{lem:qvi-contraction}
The Q-VI iterates satisfy
\[
\left\| {Q_{k+1}  - Q^* } \right\|_\infty   \le \gamma \left\| {Q_k  - Q^* } \right\|_\infty.
\]
\end{lemma}
The proof in~\cite[Lemma~2.5]{bertsekas1996neuro} is based on the contraction property of the Bellman operator. A direct consequence of~\cref{lem:qvi-contraction} is
\begin{equation}\label{eq:qvi-convergence}
\left\| {Q_k  - Q^* } \right\|_\infty   \le \gamma^k \left\| {Q_0  - Q^* } \right\|_\infty.
\end{equation}
The standard estimate in~\cref{eq:qvi-convergence} controls convergence to the fixed point \(Q^*\). The rest of the paper studies the geometry of the same recursion relative to the affine set \(\mathcal X_1\) and the POSS \(\mathcal X^*\).

\section{Switching system model of Q-VI}\label{sec:switching-model}
This section fixes the switching-system terminology and the compact notation used later. The full affine switching derivation is given in Appendix~\ref{app:qvi-switching-model}; the main text records only the elements needed for the geometric analysis.

\subsection{Switching system}
Consider the linear switching system~\cite{liberzon2003switching,lin2009stability}
\[
x_{k+1}=A_{\sigma_k} x_k,\quad x_0=z\in {\mathbb
R}^n,\quad k\in \{0,1,\ldots \},
\]
where \(x_k\in\mathbb R^n\) is the state, \(\sigma_k\in\mathcal M:=\{1,2,\ldots,M\}\) is the switching signal, and \({\cal H}:=\{A_1,A_2,\ldots,A_M\}\) is the family of subsystem matrices. Linear switching systems and their control theory are well studied~\cite{liberzon2003switching,lin2009stability}. The joint spectral radius (JSR)~\cite{tsitsiklis1997lyapunov,rota1960note,blondel2005computationally} measures the largest exponential growth rate attainable under arbitrary switching among the subsystem matrices.
\begin{definition}[Joint spectral radius~\cite{rota1960note}]\label{def:joint-spectral-radius}
Given a set of matrices \(\{ A_i \in{\mathbb R}^{n\times n} \}_{i=1}^M\), the JSR is defined as
\[
    \rho(A_1,\cdots, A_M)= \lim_{k\to\infty}\max_{\bar \sigma_k \in {\cal M}^k}\left\| A_{\sigma_k}\cdots A_{\sigma_2} A_{\sigma_1} \right\|^{1/k},
\]
where \(\|\cdot \|\) is any submultiplicative matrix norm, and \(\bar \sigma_k : = ({\sigma _1},{\sigma _2}, \ldots ,{\sigma _k}) \in {\cal M}^k\).
\end{definition}
The JSR is independent of the chosen submultiplicative norm. A related class of systems is the affine switching system
\[
x_{k+1}=A_{\sigma_k} x_k + b_{\sigma_k},\quad x_0=z\in {\mathbb
R}^n,\quad k\in \{0,1,\ldots \},
\]
where the input vector \(b_{\sigma_k}\in{\mathbb R}^n\) also switches according to \(\sigma_k\). The affine term introduces a dependence on the active mode that is absent from purely linear switching systems.

\subsection{Definitions}\label{sec:definitions}
Throughout the paper, we use the following compact notation:
\[
P:= \begin{bmatrix}
   P_1\\
   \vdots\\
   P_{|{\cal A}|}
\end{bmatrix}\in{\mathbb R}^{|{\cal S}||{\cal A}| \times |{\cal S}|},\quad
R:= \begin{bmatrix}
   R(\cdot,1) \\
   \vdots \\
   R(\cdot,|{\cal A}|) \\
\end{bmatrix}\in{\mathbb R}^{|{\cal S}||{\cal A}|},\quad
Q:= \begin{bmatrix}
   Q(\cdot,1)\\
  \vdots\\
   Q(\cdot,|{\cal A}|)
\end{bmatrix}\in {\mathbb R}^{|{\cal S}||{\cal A}|},
\]
where \(P_a=P(\cdot|\cdot,a)\in{\mathbb R}^{|{\cal S}|\times|{\cal S}|}\), \(Q(\cdot,a)\in{\mathbb R}^{|{\cal S}|}\) for \(a\in{\cal A}\), and \(R\in{\mathbb R}^{|{\cal S}||{\cal A}|}\) enumerates \(R(s,a):={\mathbb E}[r_{k+1}|s_k=s,a_k=a]\) in an order compatible with the vectorization of \(Q\). Thus, a Q-function is encoded as a single vector \(Q\in{\mathbb R}^{|{\cal S}||{\cal A}|}\) whose entries list \(Q(s,a)\) for all \(s\in{\cal S}\) and \(a\in{\cal A}\). In particular,
\[
Q(s,a) = (e_a  \otimes e_s )^\top Q,
\]
where \(e_s\in{\mathbb R}^{|{\cal S}|}\) and \(e_a\in{\mathbb R}^{|{\cal A}|}\) are the \(s\)-th and \(a\)-th basis vectors, respectively. For any stochastic policy \(\pi:{\cal S}\to \Delta_{|{\cal A}|}\), define the corresponding action transition matrix by
\begin{equation}\label{eq:policy-transition-matrix}
\Pi^\pi:=\begin{bmatrix}
   \pi(1)^\top \otimes e_1^\top\\
   \pi(2)^\top \otimes e_2^\top\\
    \vdots\\
   \pi(|{\cal S}|)^\top \otimes e_{|{\cal S}|}^\top \\
\end{bmatrix}\in {\mathbb R}^{|{\cal S}| \times |{\cal S}||{\cal A}|},
\end{equation}
where \(e_s \in \mathbb{R}^{|{\cal S}|}\). Then \(P\Pi^\pi \in \mathbb{R}^{|{\cal S}||{\cal A}| \times |{\cal S}||{\cal A}|}\) is the transition probability matrix of the state-action pair under policy \(\pi\). For a deterministic policy \(\pi:{\cal S}\to{\cal A}\), we use the corresponding one-hot vector \(\vec{\pi}(s):=e_{\pi(s)}\in\Delta_{|{\cal A}|}\), and the action transition matrix is given by~\cref{eq:policy-transition-matrix} with \(\pi\) replaced by \(\vec{\pi}\). For any \(Q\in\mathbb R^{|{\cal S}||{\cal A}|}\), let \(\pi_Q(s):=\argmax_{a\in{\cal A}}Q(s,a)\) denote the tie-broken greedy policy with respect to \(Q\). We also use the shorthand
\[
\Pi_Q:=\Pi^{\pi_Q}.
\]

\subsection{Switching-system representation of Q-VI}
With this vector notation, Q-VI can be written as a switching system without altering the recursion. In compact form, the Bellman update is
\[
Q_{k+1}=R+\gamma P\Pi_{Q_k}Q_k=FQ_k,
\]
where the active matrix is determined by the tie-broken greedy policy induced by the current iterate. For each deterministic policy \(\pi_i\in\Theta\), write
\[
A_i:=\gamma P\Pi^{\pi_i},
\qquad i=1,\ldots,M:=|\Theta|.
\]
The detailed affine switching-system representation of the error, together with the Q-VI algorithm and the basic norm and JSR facts for the full switching family, is given in Appendix~\ref{app:qvi-switching-model}; see in particular \cref{alg:qvi,lem:maxnorm-AQ,lem:full-jsr-gamma}. This formulation supplies the dynamical-systems notation used below. The geometric analysis begins with the invariant tube around \(\mathcal X_1\).
\stepcounter{lemma}
\stepcounter{lemma}

\section{Invariant tube and practically optimal solution set}\label{sec:invariant-tube}

The switching-system viewpoint summarized above, and detailed in Appendix~\ref{app:qvi-switching-model}, gives a dynamical framework for Q-VI trajectories. The analysis rests on two geometric objects: the practically optimal solution set (POSS) and an invariant tube around a distinguished affine subspace. Together they formalize the fact that Q-VI may identify the optimal action class before it reaches \(Q^*\).
The POSS is defined by
\[
\mathcal X^* :=
\left\{
Q\in\mathbb R^{|{\cal S}||{\cal A}|}:\ \pi_Q(s)\in \Phi^*(s),\ \forall s\in {\cal S}
\right\},
\]
where \(\pi_Q\) is the tie-broken greedy policy induced by \(Q\):
\[
{\pi _Q}(s): = \argmax _{a \in {\cal A}}Q(s,a).
\]
Thus, \(Q\in\mathcal X^*\) means that the induced tie-broken greedy policy \(\pi_Q\) is optimal, even if \(Q\neq Q^*\). We also define the affine space
\[
\mathcal X_1
:=
\left\{
Q\in\mathbb R^{|{\cal S}||{\cal A}|}:
Q=Q^*+\alpha \mathbf{1},\ \alpha\in\mathbb R
\right\}
=
\operatorname{span}(\mathbf{1})+Q^*.
\]
Adding a constant \(\alpha\mathbf 1\) to \(Q^*\) does not change action orderings. Therefore, the resulting tie-broken greedy policy remains optimal. Hence \(\mathcal X_1\subset\mathcal X^*\); see~\cref{lem:X1-subset-Xstar} in Appendix~\ref{app:moved-main-results}.
\stepcounter{lemma}
The affine set \(\mathcal X_1\) plays a central role because it is invariant under the Bellman operator \(F\).
\stepcounter{proposition}
This invariance property is stated and proved in~\cref{prop:X1-invariant} in Appendix~\ref{app:moved-main-results}.

It remains to ask whether \(\mathcal X^*\) itself is invariant under \(F\). In general it is not. However, a sufficiently small tube around \(\mathcal X_1=Q^*+\operatorname{span}(\mathbf 1)\) is both invariant and contained in \(\mathcal X^*\). This link between set invariance and policy identification is the key point of the section. We first exclude the degenerate case in which every action is optimal at every state.
\begin{assumption}[Optimal-class separation]\label{assump:optimal-class-separation}
Let
\[
{\cal S}_{\mathrm{sep}}
:=
\{s\in {\cal S}:\ \Phi^*(s)\neq {\cal A}\}.
\]
We assume throughout that \({\cal S}_{\mathrm{sep}}\neq\varnothing\).
\end{assumption}
\cref{assump:optimal-class-separation} guarantees that at least one state has a strict separation between optimal and non-optimal actions. For each \(s\in{\cal S}_{\mathrm{sep}}\), define
\[
\bar\Delta_s
:=
V^*(s)-\max_{a\notin \Phi^*(s)}Q^*(s,a),
\]
and let
\[
\bar\Delta
:=
\min_{s\in {\cal S}_{\mathrm{sep}}}\bar\Delta_s.
\]
Since the action space is finite, \(\bar\Delta>0\). If \({\cal S}_{\mathrm{sep}}=\varnothing\), then every action is optimal at every state, and the identification problem is trivial. The assumption is imposed only to exclude this case.
The invariant-tube statement is as follows.
\begin{proposition}[Invariant tube inside $\mathcal X^*$]\label{prop:invariant-tube}
Let \cref{assump:optimal-class-separation} hold, and fix any $\delta\in(0,\bar\Delta/2)$. Define the tube around $\mathcal X_1$ by
\[
{\cal T}_\delta
:=
\left\{
Q\in\mathbb R^{|{\cal S}||{\cal A}|}:\ \dist_\infty(Q,\mathcal X_1)\le \delta
\right\}.
\]
Then \({\cal T}_\delta \subset \mathcal X^*\). Moreover,
\[
F({\cal T}_\delta)\subset {\cal T}_{\gamma\delta}\subset {\cal T}_\delta,
\]
so \({\cal T}_\delta\) is invariant under the Bellman operator \(F\).
\end{proposition}

\noindent\emph{Proof.} See Appendix~\ref{app:proof-prop:invariant-tube}.

Consequently, once \(Q_k\) enters \({\cal T}_\delta\), every later Q-VI iterate remains in the tube. Because \({\cal T}_\delta\subset\mathcal X^*\), tube entrance also means POSS entrance. A basic finite-time entrance bound based on the standard \(\gamma\)-contraction is given in~\cref{cor:basic-entrance-Xstar} in Appendix~\ref{app:moved-main-results}.
\stepcounter{corollary}

Thus Q-VI need not reach \(Q^*\) in finite time in order for its tie-broken greedy policy to become optimal. In this sense it has a finite policy-identification property analogous to policy iteration.

The remaining issue is quantitative. Q-VI converges to \(Q^*\) at the classical rate \(\gamma\), but the approach to the invariant tube can be faster. The next section makes this precise using a restricted JSR.

The mechanism in~\cref{fig:overall-geometric-picture} underlies both the invariant-tube argument and the faster identification result. The Q-VI trajectory first moves toward \(\mathcal X_1=Q^*+\operatorname{span}(\mathbf 1)\), enters an invariant tube \({\cal T}_\delta\subset\mathcal X^*\), and thereby identifies an optimal tie-broken greedy policy in finite time. After identification, the remaining evolution continues toward \(Q^*\), possibly along a slower mode associated with the all-ones direction.

\section{Faster identification of the POSS}\label{sec:faster-identification}
We next estimate how quickly Q-VI approaches the tube \({\cal T}_\delta\). Since \(\mathcal X_1\) is the center line of this tube, it suffices to analyze how quickly the iterates approach \(\mathcal X_1\). Although the tube is defined using \(\dist_\infty\), the Euclidean distance used below is sufficient for tube entrance because
\[
\dist_\infty(Q,\mathcal X_1)
\le
\dist_2(Q,\mathcal X_1),
\qquad \forall Q\in\mathbb R^{|{\cal S}||{\cal A}|}.
\]
Thus a sufficiently small \(\ell_2\)-distance to \(\mathcal X_1\) implies membership in the corresponding \(\ell_\infty\)-tube. The first step is an exact switching-system representation of the Q-VI error over stochastic policies. The main idea behind the next lemma is adapted from the affine-map and linear time-varying viewpoint of Goyal and Grand-Cl\'ement~\cite{goyal2023first}: the Bellman maximization can be represented by selecting maps from a family, and products of those maps can then be analyzed dynamically. Here we use the same idea for Q-VI, but represent the Bellman max-difference exactly through a state-dependent stochastic policy and then project away the all-ones direction.
\begin{lemma}[Exact stochastic-policy representation of the Q-VI error]
\label{lem:exact-stochastic-representation}
Let
\[
e_k := Q_k - Q^*\in{\mathbb R}^{|{\cal S}||{\cal A}|}.
\]
Then, for each \(k\ge 0\), there exists a stochastic policy \(\mu_k:{\cal S}\to\Delta_{|{\cal A}|}\) such that
\[
e_{k+1} = A_{\mu_k} e_k,
\]
where \(A_\mu := \gamma P\Pi^\mu\in{\mathbb R}^{|{\cal S}||{\cal A}|\times |{\cal S}||{\cal A}|}\).
\end{lemma}

\noindent\emph{Proof.} See Appendix~\ref{app:proof-lem:exact-stochastic-representation}.

The stochastic policy \(\mu_k\) in \cref{lem:exact-stochastic-representation} is only an auxiliary representation of the Bellman max-difference as a convex combination of current errors. It is generally not the greedy policy \(\pi_{Q_k}\), need not be unique, and is not implemented by the algorithm.

Thus the Q-VI error admits an exact linear switching representation; no approximation is involved. The maximization in the Bellman operator is absorbed into the state-dependent stochastic policy \(\mu_k\). This permits the use of spectral and Lyapunov tools for the linear maps \(A_\mu\) when studying convergence toward \(\mathcal X_1\).

Because our main object is the distance from \(Q_k\) to
\[
\mathcal X_1 = Q^* + \mathrm{span}(\mathbf{1}),
\]
we isolate the component of the error orthogonal to \(\mathrm{span}(\mathbf 1)\). The component along \(\mathbf 1\) only shifts the Q-function uniformly and does not affect the tie-broken greedy policy. Let
\[
\proj := I-\frac{1}{n}\mathbf{1}\mathbf{1}^\top,
\qquad
n:=|{\cal S}||{\cal A}|,
\]
be the orthogonal projection onto \(\mathrm{span}(\mathbf{1})^\perp\). Define the projected error
\[
z_k := \proj e_k\in{\mathbb R}^{n},
\]
the restricted matrices
\[
\bar A_i = \proj A_i \proj,\quad i\in \{1,2,\ldots,M\},
\]
and, for each stochastic policy \(\mu\),
\[
\bar A_\mu := \proj A_\mu \proj.
\]
We denote the deterministic restricted family and its JSR by
\[
\bar{\cal H}:=\{\bar A_1,\bar A_2,\ldots,\bar A_M\},
\qquad
\bar\rho:=\rho(\bar A_1,\bar A_2,\ldots,\bar A_M).
\]
The projected error then evolves under the restricted stochastic-policy family, and its norm is exactly the distance to \(\mathcal X_1\).
\begin{lemma}[Projected error dynamics]
\label{lem:projected-error-dynamics}
The projected error \(z_k\) satisfies
\[
z_{k+1}=\bar A_{\mu_k} z_k
\qquad \forall k\ge 0.
\]
Moreover, \(\dist_2(Q_k,\mathcal X_1)=\|z_k\|_2\).
\end{lemma}

\noindent\emph{Proof.} See Appendix~\ref{app:proof-lem:projected-error-dynamics}.

Several auxiliary facts about the restricted matrices are needed to analyze the projected dynamics in~\cref{lem:projected-error-dynamics}. Specifically, \cref{lem:restricted-matrix-properties} records how the projection removes the all-ones direction, \cref{lem:projection-product-identity} gives the corresponding product identity, \cref{lem:restricted-spectral-radius} characterizes the spectral radius of each restricted subsystem, and \cref{lem:jsr-upper-bound} gives a norm-based computable upper bound on the restricted JSR. These auxiliary results are stated and proved in Appendix~\ref{app:moved-main-results}.
\stepcounter{lemma}
\stepcounter{lemma}
\stepcounter{lemma}
\stepcounter{lemma}

\begin{lemma}\label{lem:restricted-jsr-upper-gamma}
The JSR of the restricted switching system satisfies
\[
\rho(\bar A_1,\ldots,\bar A_M)\le \gamma.
\]
\end{lemma}

\noindent\emph{Proof.} See Appendix~\ref{app:proof-lem:restricted-jsr-upper-gamma}.

Hence the restricted JSR never exceeds the classical rate \(\gamma\). An explicit exponential estimate follows from a common Lyapunov norm for the restricted family, together with an extension of the Lyapunov inequality to the convex hull. These two ingredients are stated and proved in~\cref{lem:common-lyapunov-restricted-family,lem:convex-hull-extension} in Appendix~\ref{app:moved-main-results}; together, they allow the Lyapunov inequality to be applied directly to the stochastic-policy projected dynamics.
\stepcounter{lemma}
\stepcounter{lemma}

This gives exponential convergence of the actual Q-VI iterates toward \(\mathcal X_1\).
\begin{theorem}[Global exponential convergence to $\mathcal X_1$ for the actual Q-VI]
\label{thm:global-exp-conv-X1}
Fix any $\epsilon>0$ such that
\[
\beta_\epsilon:=\bar\rho+\epsilon \in (0,1).
\]
There exists a constant \(c_\varepsilon>0\), depending on \(\varepsilon>0\), such that
\begin{equation}\label{eq:global-exp-conv-X1}
\dist_2(Q_k,\mathcal X_1)
\le
c_\varepsilon \beta_\varepsilon^k
\dist_2(Q_0,\mathcal X_1),
\qquad
\forall k\ge 0.
\end{equation}
\end{theorem}

\noindent\emph{Proof.} See Appendix~\ref{app:proof-thm:global-exp-conv-X1}.

This theorem gives exponential convergence to \(\mathcal X_1\). Since \(\bar\rho\le\gamma\), the case \(\bar\rho<\gamma\) permits a choice of \(\varepsilon>0\) with \(\beta_\varepsilon<\gamma\). The transverse approach to \(\mathcal X_1\) is then strictly faster than the standard Q-VI rate for convergence to \(Q^*\).

The same estimate gives an explicit finite-time bound for entrance into the POSS; see~\cref{cor:fast-identification-pos} in Appendix~\ref{app:moved-main-results}. If \(\bar\rho<\gamma\), then \(\varepsilon\) can be chosen so that \(\beta_\varepsilon<\gamma\), and the asymptotic transverse rate is strictly faster than the classical \(\gamma\)-rate. The resulting finite-time bound may be smaller than the basic \(\gamma\)-based estimate, depending on the constant \(c_\varepsilon\) and the initial transverse distance \(\dist_2(Q_0,\mathcal X_1)\).
\stepcounter{corollary}

\cref{cor:fast-identification-pos} gives finite-time identification of the POSS through convergence toward \(\mathcal X_1\). Once the iterate is close enough to \(\mathcal X_1\), its tie-broken greedy policy selects only optimal actions, even though the full Q-function has not yet converged to \(Q^*\). When \(\beta_\varepsilon<\gamma\), this provides a potentially sharper policy-identification estimate than the classical contraction argument.

The geometry is therefore two-stage. First, \(Q_k\) approaches \(\mathcal X_1\), enters \({\cal T}_\delta\), and hence enters \(\mathcal X^*\). From then on, every tie-broken greedy policy \(\pi_{Q_k}\) is optimal, so the remaining iterations amount to evaluation over optimal policies. The final convergence to \(Q^*\) may still be limited by the standard \(\gamma\)-rate.

\section{Two-stage convergence}
The finite-time identification bound in~\cref{cor:fast-identification-pos} gives a two-stage description of the Q-VI trajectory. The full statement is given in~\cref{thm:two-stage-convergence} in Appendix~\ref{app:moved-main-results}. Before identification, the transverse distance to \(\mathcal X_1\) admits exponential upper bounds at any rate larger than the full restricted JSR \(\bar\rho\). After identification, only optimal policies remain active, and the transverse component admits analogous bounds at any rate larger than the JSR \(\bar\rho_*\) of the smaller optimal restricted family. If \(\bar\rho_*<\gamma\), then the post-identification transverse decay is strictly faster than the standard \(\gamma\)-rate, although the full error may still be dominated by the residual all-ones component.

When the optimal policy is unique, the post-identification switching disappears. This sharper special case is stated in~\cref{cor:unique-optimal-policy} in Appendix~\ref{app:moved-main-results}; it reduces the second stage to a single linear system associated with \(\pi^*\), whose transverse rate is \(\rho(\bar A_{\pi^*})=\gamma |\lambda_2(P\Pi^{\pi^*})|\).
\stepcounter{theorem}
\stepcounter{corollary}

The numerical examples below show the mechanism for deterministic Q-VI and for tabular Q-learning. The Q-VI experiment demonstrates finite-time entrance into the POSS through approach to \(\mathcal X_1\), while the Q-learning experiment suggests that the same geometric pattern can persist under stochastic updates.

\section{Structural conditions for \texorpdfstring{$\bar\rho<\gamma$}{rho-bar < gamma}}\label{sec:strict-jsr-conditions}
The strict inequality \(\bar\rho<\gamma\) is tied to uniform mixing of the projected state-action dynamics in directions transverse to \(\operatorname{span}(\mathbf 1)\). The detailed formulation is deferred to Appendix~\ref{app:strict-jsr-conditions}. There, \cref{prop:uniform-scrambling-strict} characterizes strictness through a uniform scrambling condition for deterministic policy products, using scrambling matrices, Dobrushin coefficients, and the diameter seminorm introduced in~\cref{def:scrambling-matrix,def:dobrushin-coefficient,def:Diameter-seminorm}. The supporting contraction facts are collected in~\cref{lem:dobrushin-Diameter-properties}. The appendix also records transparent sufficient conditions, such as common-descendant and Doeblin-type overlap assumptions, together with the standard Markov-chain obstruction caused by nontrivial unit-modulus modes of deterministic policies. In particular, single-policy ergodicity of every \(P_\pi\) is necessary but not sufficient for strictness of the full switched family; arbitrary switching requires the bounded-scrambling condition.

\section{Examples}\label{sec:examples}
This section gives a short summary of the numerical examples; the complete MDP specification, numerical values, and discussion are in Appendix~\ref{app:examples}. The appendix considers a small discounted MDP with $\mathcal S=\{1,2,3\}$, $\mathcal A=\{1,2\}$, and $\gamma=0.95$; for this MDP, the unique optimal deterministic policy is $\pi^*=\bigl(\pi^*(1),\pi^*(2),\pi^*(3)\bigr)=(1,1,1)$. It also includes a tabular Q-learning experiment on the same problem. The figures in this section reproduce the plots from Appendix~\ref{app:examples}.

For Q-VI, the normalized decay plot in \cref{fig:vi-normalized-main} shows that \(\dist_2(Q_k,\mathcal X_1)\) decreases faster than \(\|Q_k-Q^*\|_\infty\) in the transient regime. The projected single-trajectory and multi-trajectory plots in \cref{fig:vi-projection-single-main,fig:vi-projection-multiple-main} show further that the iterates are rapidly drawn toward the affine set \(\mathcal X_1\), enter the tube \({\cal T}_\delta\) in finite time, and then converge to \(Q^*\). The displayed single Q-VI trajectory is used to visualize tube entrance; its initial tie-broken greedy policy is already optimal, so it should not be interpreted as a first-time policy-identification example from a non-optimal greedy policy. For tabular Q-learning, \cref{fig:ql-projection-main} shows the analogous stochastic behavior: despite sample-path fluctuations, trajectories from different initial points tend to approach the neighborhood of \(\mathcal X_1\) first and only afterward refine their convergence toward \(Q^*\).

\begin{figure}[t]
\centering
\includegraphics[width=0.62\linewidth]{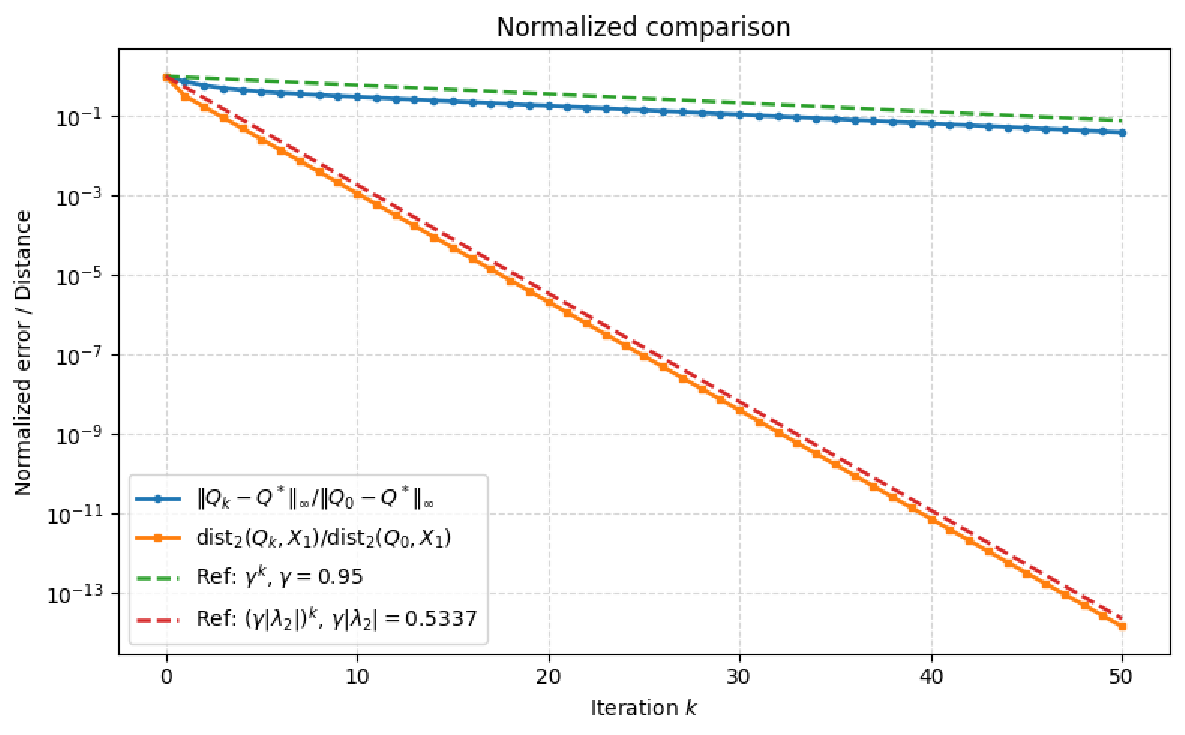}
\caption{Normalized comparison of
\(
\|Q_k-Q^*\|_\infty
\)
and
\(
\dist_2(Q_k,\mathcal X_1)
\)
for the toy MDP. The dashed curves indicate the reference rates
\(
\gamma^k
\)
and
\(
(\gamma|\lambda_2|)^k
\).
The plot shows that the iterate approaches the affine set \(\mathcal X_1\) faster than it
approaches the optimal Q-function \(Q^*\).}
\label{fig:vi-normalized-main}
\end{figure}

\begin{figure}[t]
\centering
\includegraphics[width=0.50\linewidth]{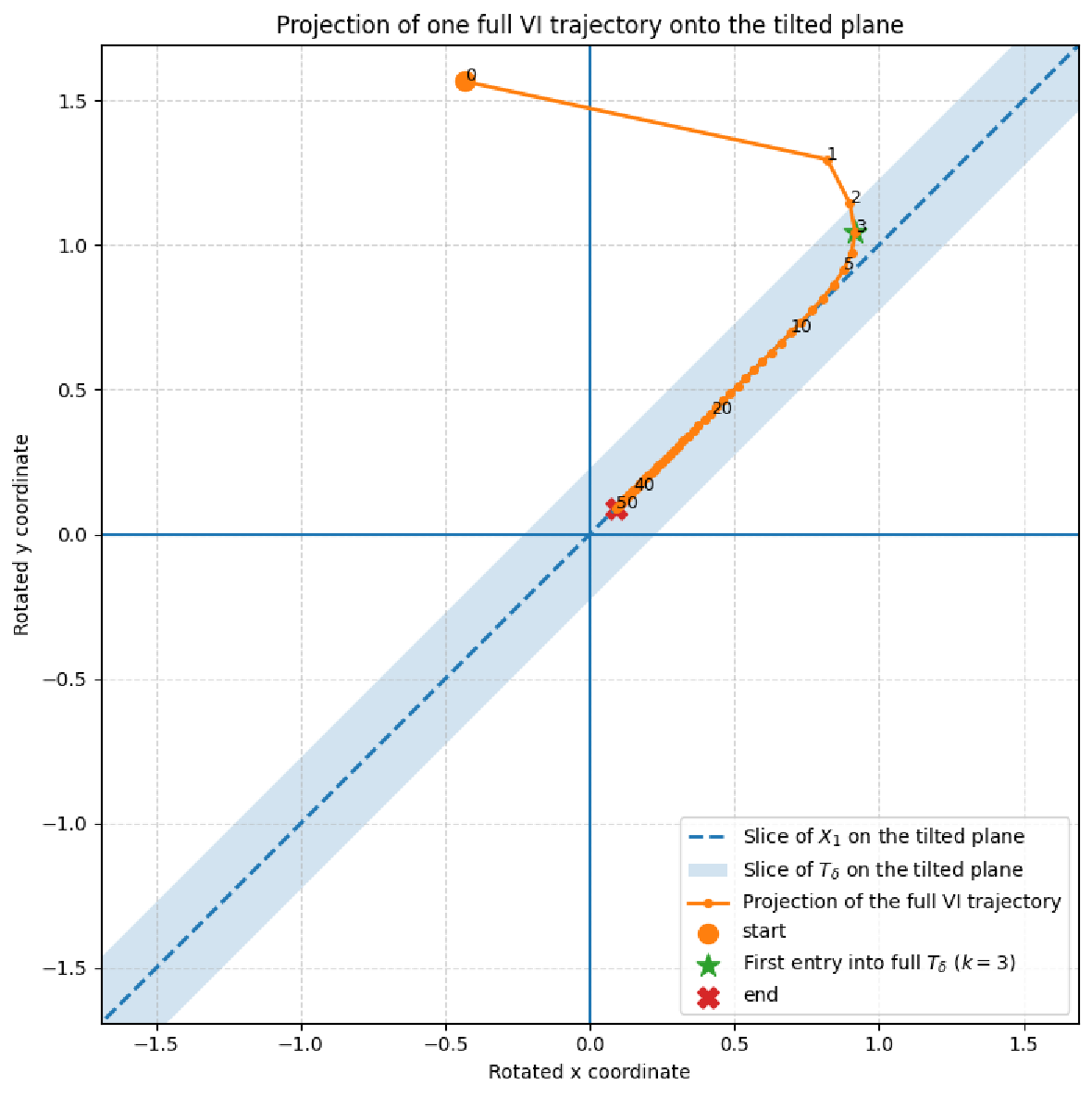}
\caption{Orthogonal projection of a single full Q-VI trajectory onto the tilted plane.
The dashed line is the slice of \(\mathcal X_1\) and the shaded strip is the slice of
the tube \({\cal T}_\delta\).
The projected trajectory starts from the initial point, enters the strip, and then converges to the origin.}
\label{fig:vi-projection-single-main}
\end{figure}

\begin{figure}[t]
\centering
\includegraphics[width=0.50\linewidth]{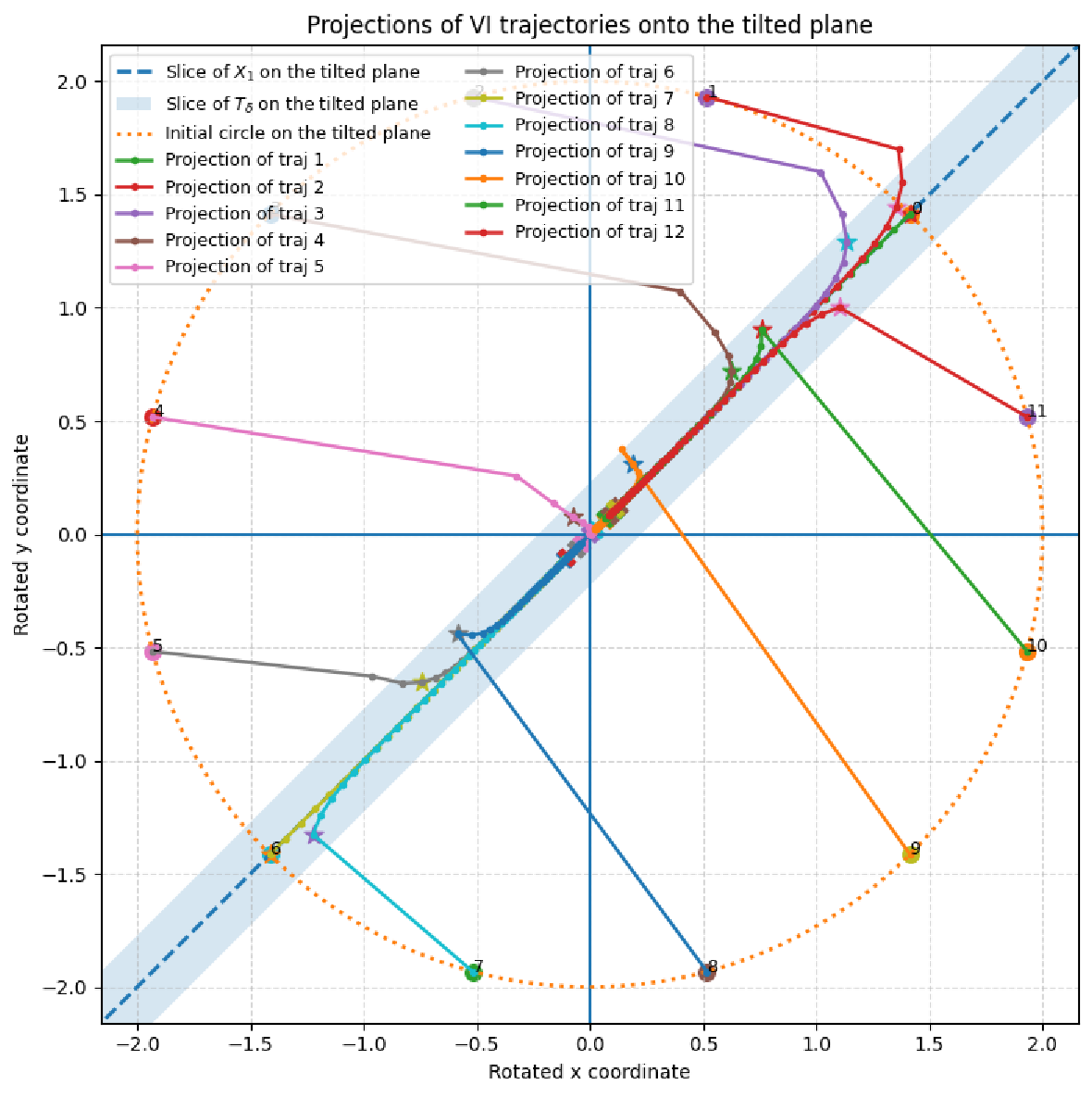}
\caption{Orthogonal projections of \(12\) full Q-VI trajectories onto the tilted plane,
where the initial conditions are chosen uniformly on a circle of radius \(2\) in the plane.
The dashed line is the slice of \(\mathcal X_1\), the shaded strip is the slice of \({\cal T}_\delta\), and the
dotted curve is the initial circle.
The plot gives a global geometric view of trajectories being first attracted toward
\(\mathcal X_1\) and then converging to \(Q^*\).}
\label{fig:vi-projection-multiple-main}
\end{figure}

\begin{figure}[t]
\centering
\includegraphics[width=0.60\linewidth]{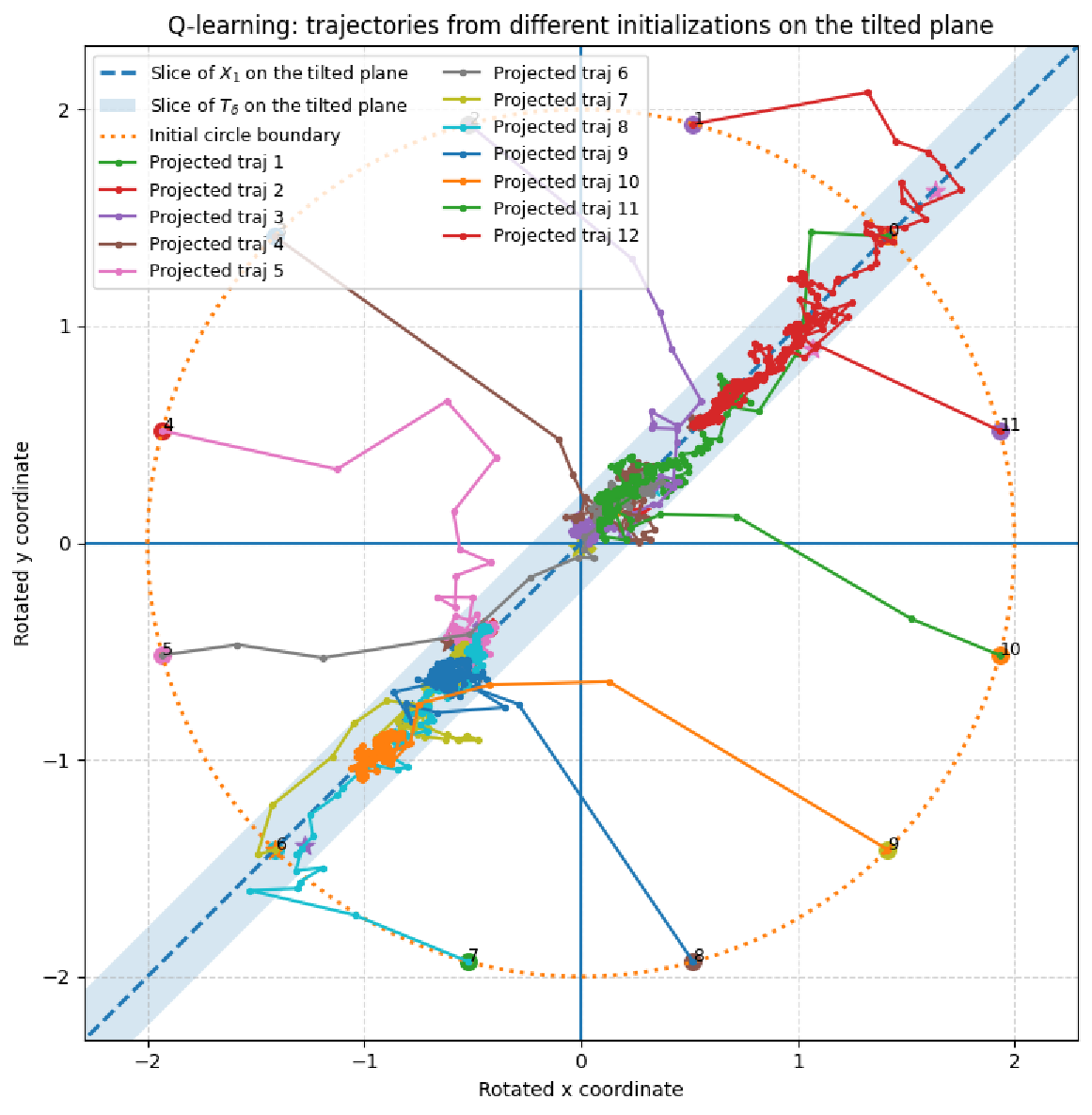}
\caption{Two-dimensional projection of tabular Q-learning trajectories onto the tilted plane
\(Q^*+\operatorname{span}\{\hat{\mathbf 1},\hat d\}\), displayed in the rotated coordinates
\(p=(u-v)/\sqrt{2}\) and \(q=(u+v)/\sqrt{2}\).
The dashed line \(q=p\) is the projection of \(\mathcal X_1\), and the shaded strip is the projection
of the tube \({\cal T}_\delta\), namely \( |q-p|\le 2\delta \).
The dotted circle indicates the boundary of the common initial set in the \((u,v)\)-plane.
Starting from \(12\) different initial points on this circle, the stochastic Q-learning
trajectories tend to enter the strip and then move toward the origin, which corresponds to
\(Q^*\).}
\label{fig:ql-projection-main}
\end{figure}

\section{Conclusion}\label{sec:conclusion}
We studied discounted Q-VI from a switching-system perspective and proved that Q-VI identifies the optimal action class in finite time, before asymptotic convergence to \(Q^*\) is necessarily complete. The mechanism is the approach to the affine set \(\mathcal X_1=Q^*+\operatorname{span}(\mathbf 1)\) and entrance into an invariant tube contained in the POSS. This approach admits exponential upper bounds at any rate larger than the restricted JSR and can be faster than the classical \(\gamma\)-rate when the restricted JSR is strictly smaller than \(\gamma\). We also gave structural conditions, based on spectral and graph-theoretic arguments, that explain when this strict improvement occurs and when Markov-chain obstructions prevent it. The resulting two-stage description separates fast policy identification from final convergence to the Bellman fixed point.

\bibliographystyle{IEEEtran}
\bibliography{reference}

\appendix

\section{Q-VI algorithm and switching-system representation}\label{app:qvi-switching-model}
This appendix gives the algorithmic and switching-system details used in the main text. Algorithm~\ref{alg:qvi} expands the compact recursion in~\cref{eq:qvi-basic}, and the subsequent derivation shows how the same update can be written as an affine switching system whose active mode is determined by the tie-broken greedy policy.

\subsection{Q-VI algorithm}\label{app:qvi-algorithm}
\begin{algorithm}[h]
\caption{Q-VI}\label{alg:qvi}
  \begin{algorithmic}[1]
    \State Initialize $Q_0 \in {\mathbb R}^{|{\cal S}||{\cal A}|}$ arbitrarily.
    \For{iteration $k=0,1,\ldots$}
        \State Update
        \[
Q_{k + 1} (s,a) = \underbrace {R(s,a) + \gamma \sum\limits_{s' \in {\cal S}} {P(s'|s,a)\max _{a' \in {\cal A}} Q_k (s',a')}}_{ = :F Q_k }
\]
    \EndFor
  \end{algorithmic}
\end{algorithm}
Algorithm~\ref{alg:qvi} is exactly the componentwise form of \(Q_{k+1}=FQ_k\) in~\cref{eq:qvi-basic}. The componentwise form above makes explicit how the maximization over next-state actions enters the Bellman update.

\subsection{Switching-system model of Q-VI}\label{app:switching-model-qvi}
With the notation and the Bellman update in place, we reformulate the iteration as a switching system. The reformulation makes explicit that the iterate is governed by affine subsystems whose active mode is determined by the tie-broken greedy policy of the current \(Q\)-function.

Using the notation introduced in the main text, the update in Algorithm~\ref{alg:qvi} can be rewritten as
\begin{equation}\label{eq:qvi-bellman-vector}
Q_{k+1}= R+\gamma P\Pi_{Q_k}Q_k:=F(Q_k).
\end{equation}
Recall the definitions of $\pi_Q(s)$ and $\Pi_Q$. Invoking the optimal Bellman equation $(\gamma P\Pi_{Q^*}-I)Q^*+R=0$,~\cref{eq:qvi-bellman-vector} can be further rewritten as
\[
(Q_{k + 1}  - Q^* ) = \gamma P\Pi _{Q_k } (Q_k  - Q^* ) + \gamma P(\Pi _{Q_k }  - \Pi _{Q^* } )Q^*.
\]
Thus, in error coordinates, Q-VI is an affine switching system. In particular, for any $Q \in {\mathbb R}^{|{\cal S}||{\cal A}|}$, define
\[
A_Q : = \gamma P\Pi _Q \in {\mathbb R}^{|{\cal S}||{\cal A}|\times |{\cal S}||{\cal A}|},\quad b_Q : = \gamma P(\Pi _Q  - \Pi _{Q^* } )Q^* \in {\mathbb R}^{|{\cal S}||{\cal A}|}.
\]
Then Q-VI can be represented as
\begin{equation}\label{eq:qvi-switched-affine-error}
Q_{k + 1}- Q^* = A_{Q_k} (Q_k - Q^*) + b_{Q_k},
\end{equation}
where $A_{Q_k}$ and $b_{Q_k}$ switch among matrices from $\{\gamma P\Pi^\pi:\pi\in \Theta\}$ and vectors from $\{\gamma P(\Pi^\pi - \Pi^{\pi^*})Q^* :\pi\in\Theta \}$ according to the changes of $Q_k$.

To make this switching description finite-dimensional in the usual switching-system notation, let $\varphi :\Theta  \to \{ 1,2, \ldots ,|\Theta |\}$ be a one-to-one mapping from a deterministic policy $\pi \in \Theta$ to an integer in $\{ 1,2, \ldots ,|\Theta |\}$. For $i = \varphi(\pi)$, define
\[
{A_i} = \gamma P{\Pi ^\pi } \in {\mathbb R}^{|{\cal S}||{\cal A}| \times |{\cal S}||{\cal A}|},
\qquad
{b_i} = \gamma P(\Pi^\pi - \Pi^{\pi^*})Q^* \in {\mathbb R}^{|{\cal S}||{\cal A}|}.
\]
Then~\cref{eq:qvi-switched-affine-error} is an affine switching system with switching signal ${\sigma _k} \in \{ 1,2, \ldots ,|\Theta |\}$ determined by ${\sigma _k} = \varphi ({\pi _k})$, where
\[
{\pi _k}(\cdot): = \arg {\max _{a \in {\cal A}}}{Q_k}( \cdot ,a) \in \Theta.
\]
Two elementary facts connect this affine switching model with the classical contraction rate. The first gives a uniform norm bound for every active subsystem, and the second identifies the JSR of the full switching family.

\setcounter{lemma}{1}
\begin{lemma}[\cite{lee2023discrete}]\label{lem:maxnorm-AQ}
For any $Q \in {\mathbb R}^{|{\cal S}||{\cal A}|}$, $\|A_Q \|_\infty = \gamma$, where the matrix norm  $\| A \|_\infty :=\max_i \sum_j {|A_{ij} |}$ and $A_{ij}$ is the element of $A$ in the $i$-th row and $j$-th column.
\end{lemma}
\begin{proof}
Note that $\sum_j|[A_Q]_{ij}|=\sum_j {| [\gamma P\Pi _Q ]_{ij}|}= \gamma$, which completes the proof.
\end{proof}

This bound also identifies the exact JSR of the linear part of the full switching family.
\begin{lemma}\label{lem:full-jsr-gamma}
The JSR of the linear part \(\{A_i\}_{i=1}^M\) of the affine switching representation~\cref{eq:qvi-switched-affine-error} is $\gamma$.
\end{lemma}
\begin{proof}
Since each $P\Pi^\pi$ is row-stochastic, for every switching sequence $\bar\sigma_k=(\sigma_1,\ldots,\sigma_k)$, the matrix
\[
A_{\sigma_k}\cdots A_{\sigma_1}
=
\gamma^k (P\Pi^{\pi_{\sigma_k}})\cdots(P\Pi^{\pi_{\sigma_1}})
\]
has infinity norm equal to $\gamma^k$. Hence,
\[
\rho ({A_1}, \cdots ,{A_M})
=
\lim_{k\to\infty}
\max_{\bar\sigma_k\in\{1,\ldots,M\}^k}
\left\|A_{\sigma_k}\cdots A_{\sigma_1}\right\|_\infty^{1/k}
=
\lim_{k\to\infty}(\gamma^k)^{1/k}
=
\gamma.
\]
This completes the proof.
\end{proof}
\setcounter{lemma}{13}

\section{Structural conditions for \texorpdfstring{$\bar\rho<\gamma$}{rho-bar < gamma}}\label{app:strict-jsr-conditions}
The preceding analysis established that the restricted JSR always satisfies
\(\bar\rho\le \gamma\). We now give structural conditions under which the
inequality becomes strict. These conditions are not imposed in the main results,
but they clarify when the projected state-action dynamics possess uniform
mixing in directions transverse to \(\operatorname{span}(\mathbf 1)\). The
analysis is based on standard coefficients of ergodicity, scrambling matrix
arguments, and span-seminorm contractions for products of stochastic
matrices~\cite{dobrushin1956central,hajnal1958weak,wolfowitz1963products,seneta2006non}.

One subtle point is that the exact Q-VI error representation in
\cref{lem:exact-stochastic-representation} is written in terms of stochastic
policies. Nevertheless, the strict inequality \(\bar\rho<\gamma\) can be
characterized using deterministic policy products. The reason is that the
stochastic-policy matrices form the convex hull of the deterministic-policy
matrices, and the JSR is invariant under convexification~\cite{rota1960note,blondel2005computationally}.

We denote the state-action index set in this section by
\[
\mathcal Y:=\mathcal S\times\mathcal A.
\]
For each deterministic policy \(\pi\in\Theta\), define
\[
B_\pi:=P\Pi^\pi,
\]
so that \(A_\pi=\gamma B_\pi\). Then, \(B_\pi\) is row-stochastic on the finite
index set \(\mathcal Y\). For each stochastic policy
\(\mu:\mathcal S\to\Delta_{|\mathcal A|}\), define similarly
\[
B_\mu:=P\Pi^\mu.
\]
Then, \(B_\mu\) is also row-stochastic. Moreover, every stochastic-policy matrix
is a convex combination of deterministic-policy matrices:
\[
B_\mu
=
\sum_{\pi\in\Theta}c_\pi(\mu)B_\pi,
\qquad
c_\pi(\mu):=\prod_{s\in\mathcal S}\mu(\pi(s)|s),
\]
where $c_\pi(\mu)\ge 0$ and $\sum_{\pi\in\Theta}c_\pi(\mu)=1$. This proves
\[
\{B_\mu:\mu:\mathcal S\to\Delta_{|\mathcal A|}\}
\subseteq
\operatorname{co}\{B_\pi:\pi\in\Theta\}.
\]
The reverse inclusion also holds. Indeed, let
\(\sum_{\pi\in\Theta}c_\pi B_\pi\) be an arbitrary convex combination, with
\(c_\pi\ge0\) and \(\sum_{\pi\in\Theta}c_\pi=1\). Define a stochastic policy
\(\mu_c:\mathcal S\to\Delta_{|\mathcal A|}\) by
\[
\mu_c(a|s)
:=
\sum_{\pi\in\Theta:\,\pi(s)=a}c_\pi,
\qquad s\in\mathcal S,
\quad a\in\mathcal A .
\]
Then, row by row in the definition of the action transition matrix,
\[
\Pi^{\mu_c}
=
\sum_{\pi\in\Theta}c_\pi\Pi^\pi.
\]
Consequently,
\[
B_{\mu_c}
=
P\Pi^{\mu_c}
=
\sum_{\pi\in\Theta}c_\pi P\Pi^\pi
=
\sum_{\pi\in\Theta}c_\pi B_\pi .
\]
Therefore,
\[
\{B_\mu:\mu:\mathcal S\to\Delta_{|\mathcal A|}\}
=
\operatorname{co}\{B_\pi:\pi\in\Theta\}.
\]
Let
\[
\widehat B_\pi:=\proj B_\pi\proj,
\qquad
\bar B_\mu:=\proj B_\mu\proj .
\]
Then, it follows that
\[
\bar B_\mu
=
\sum_{\pi\in\Theta}c_\pi(\mu)\bar B_\pi,
\]
and by the convex-hull invariance of the JSR,
\[
\rho(\bar B_\mu:\mu:\mathcal S\to\Delta_{|\mathcal A|})
=
\rho(\bar B_\pi:\pi\in\Theta).
\]
Since
\[
\bar A_\pi=\proj A_\pi\proj=\gamma\proj B_\pi\proj=\gamma\bar B_\pi,
\]
we have
\[
\bar\rho
=
\rho(\bar A_\pi:\pi\in\Theta)
=
\gamma\rho(\bar B_\pi:\pi\in\Theta).
\]

For a deterministic policy sequence
\[
\omega=(\pi_0,\pi_1,\ldots,\pi_{L-1})\in\Theta^L,
\]
define the corresponding product
\[
B_\omega:=B_{\pi_{L-1}}\cdots B_{\pi_1}B_{\pi_0}.
\]
Equivalently, the directed graph of \(B_\pi\) has an edge
\[
(s,a)\longrightarrow_\pi (s',a')
\]
if and only if \(P(s'|s,a)>0\) and \(a'=\pi(s')\). For a policy sequence
\(\omega\) and \(x\in\mathcal Y\), let
\[
\mathcal R_\omega(x)
:=
\{y\in\mathcal Y:(B_\omega)_{xy}>0\}
\]
be the set of state-action nodes reachable from \(x\) after following the
policy sequence \(\omega\).

\begin{definition}[Scrambling matrix~\cite{hajnal1958weak,wolfowitz1963products,seneta2006non}]
\label{def:scrambling-matrix}
A row-stochastic matrix \(B\in\mathbb R^{|\mathcal Y|\times|\mathcal Y|}\) is
called scrambling if every pair of rows has at least one common positive
column. Equivalently, for every \(x,y\in\mathcal Y\), there exists
\(\upsilon\in\mathcal Y\) such that
\[
B_{x\upsilon}>0
\qquad\text{and}\qquad
B_{y\upsilon}>0 .
\]
For a product \(B_\omega\), this is equivalent to
\[
\mathcal R_\omega(x)\cap\mathcal R_\omega(y)\neq\varnothing,
\qquad
\forall x,y\in\mathcal Y .
\]
\end{definition}

\begin{definition}[Dobrushin coefficient~\cite{dobrushin1956central,gaubert2015dobrushin}]
\label{def:dobrushin-coefficient}
For a row-stochastic matrix \(B\in\mathbb R^{|\mathcal Y|\times|\mathcal Y|}\),
its Dobrushin coefficient is defined by
\[
\tau_{\rm D}(B)
:=
\frac12
\max_{x,y\in\mathcal Y}
\sum_{\upsilon\in\mathcal Y}
\left|B_{x\upsilon}-B_{y\upsilon}\right|.
\]
Equivalently,
\[
\tau_{\rm D}(B)
=
1-
\min_{x,y\in\mathcal Y}
\sum_{\upsilon\in\mathcal Y}
\min\{B_{x\upsilon},B_{y\upsilon}\}.
\]
\end{definition}

\begin{definition}[Diameter seminorm~\cite{gaubert2015dobrushin,tsitsiklis1986distributed}]
\label{def:Diameter-seminorm}
For a vector \(v\in\mathbb R^{|\mathcal Y|}\), the diameter seminorm is defined as
\[
\|v\|_{\rm dm}
:=
\max_{x\in\mathcal Y}v_x
-
\min_{x\in\mathcal Y}v_x .
\]
It is a seminorm on \(\mathbb R^{|\mathcal Y|}\), and it becomes a norm on the
quotient space
\(\mathbb R^{|\mathcal Y|}/\operatorname{span}(\mathbf 1)\), since
\(\|v\|_{\rm dm}=0\) if and only if \(v\in\operatorname{span}(\mathbf 1)\).
\end{definition}

We use the following standard properties of the Dobrushin coefficient and the
diameter seminorm.
\setcounter{lemma}{13}
\begin{lemma}[Dobrushin coefficient and diameter-seminorm contraction~\cite{dobrushin1956central,hajnal1958weak,seneta2006non}]
\label{lem:dobrushin-Diameter-properties}
Let \(B\) and \(C\) be row-stochastic matrices of size
\(|\mathcal Y|\times|\mathcal Y|\). Then the following statements hold:
\begin{enumerate}
    \item \(0\le \tau_{\rm D}(B)\le 1\).
    \item \(\tau_{\rm D}(B)<1\) if and only if \(B\) is scrambling.
    \item \(\tau_{\rm D}(BC)\le \tau_{\rm D}(B)\tau_{\rm D}(C)\).
    \item If \(B_i\) are row-stochastic and \(\alpha_i\ge 0\),
    \(\sum_i\alpha_i=1\), then
    \[
    \tau_{\rm D}\!\left(\sum_i\alpha_iB_i\right)
    \le
    \sum_i\alpha_i\tau_{\rm D}(B_i).
    \]
    In particular, \(\tau_{\rm D}\) is convex on the set of row-stochastic
    matrices.
    \item The Dobrushin coefficient is the induced operator norm of \(B\) on
    the quotient space
    \(\mathbb R^{|\mathcal Y|}/\operatorname{span}(\mathbf 1)\) equipped with
    the diameter seminorm:
    \[
    \tau_{\rm D}(B)
    =
    \sup_{v\notin\operatorname{span}(\mathbf 1)}
    \frac{\|Bv\|_{\rm dm}}{\|v\|_{\rm dm}}.
    \]
    Consequently, for every \(v\in\mathbb R^{|\mathcal Y|}\),
    \[
    \|Bv\|_{\rm dm}\le \tau_{\rm D}(B)\|v\|_{\rm dm}.
    \]
    \item For every \(v\in\mathbb R^{|\mathcal Y|}\) and every \(c\in\mathbb R\),
    \[
    \|v+c\mathbf 1\|_{\rm dm}=\|v\|_{\rm dm}.
    \]
    Consequently, if \(\proj=I-\frac{1}{|\mathcal Y|}\mathbf 1\mathbf 1^\top\),
    then
    \[
    \sup_{v\notin\operatorname{span}(\mathbf 1)}
    \frac{\|\proj B\proj v\|_{\rm dm}}{\|v\|_{\rm dm}}
    =
    \tau_{\rm D}(B).
    \]
\end{enumerate}
\end{lemma}

\noindent\emph{Proof.} See Appendix~\ref{app:proof-lem:dobrushin-Diameter-properties}.

We can now state an if-and-only-if characterization of strictness for
the full restricted switching family.
\setcounter{proposition}{2}
\begin{proposition}[Characterization of \(\bar\rho<\gamma\) by uniform scrambling]
\label{prop:uniform-scrambling-strict}
The following statements are equivalent:
\begin{enumerate}
    \item \(\bar\rho<\gamma\).
    \item There exists an integer \(L\ge 1\) such that, for every deterministic
    policy sequence \(\omega=(\pi_0,\ldots,\pi_{L-1})\in\Theta^L\), the product
    \(B_\omega\) is scrambling.
    \item There exists an integer \(L\ge 1\) such that
    \[
    \max_{\omega\in\Theta^L}\tau_{\rm D}(B_\omega)<1.
    \]
    \item Equivalently, there exists an integer \(L\ge 1\) such that
    \[
    \sup_{\mu_0,\ldots,\mu_{L-1}}
    \tau_{\rm D}
    \left(
    B_{\mu_{L-1}}\cdots B_{\mu_1}B_{\mu_0}
    \right)
    <1,
    \]
    where the supremum is taken over all stochastic policies
    \(\mu_0,\ldots,\mu_{L-1}:\mathcal S\to\Delta_{|\mathcal A|}\).
\end{enumerate}
Moreover, if these equivalent conditions hold and
\[
\eta_L
:=
\min_{\omega\in\Theta^L}
\min_{x,y\in\mathcal Y}
\sum_{\upsilon\in\mathcal Y}
\min\{(B_\omega)_{x\upsilon},(B_\omega)_{y\upsilon}\},
\]
then \(\eta_L>0\) and
\[
\bar\rho
\le
\gamma(1-\eta_L)^{1/L}
<
\gamma.
\]
\end{proposition}

\noindent\emph{Proof.} See Appendix~\ref{app:proof-prop:uniform-scrambling-strict}.

\cref{prop:uniform-scrambling-strict} shows that the bounded scrambling
condition is a necessary and sufficient characterization of \(\bar\rho<\gamma\) for the full restricted switching
family. It can be stronger
than what is necessary for a particular Q-VI trajectory, since not every
admissible stochastic switching sequence must be realized by the Bellman
maximization dynamics along a given initialization.

A simple sufficient condition for the equivalent properties in~\cref{prop:uniform-scrambling-strict} is a one-step common descendant. Suppose
that there exists a state \(s'\in\mathcal S\) such that
\[
p_{\min}:=\min_{s\in\mathcal S,\ a\in\mathcal A}P(s'|s,a)>0.
\]
Then, for every deterministic policy \(\pi\), all rows of \(B_\pi\) assign at
least \(p_{\min}\) mass to the common state-action column
\((s',\pi(s'))\). Hence, \(B_\pi\) is one-step scrambling and
\[
\bar\rho\le \gamma(1-p_{\min})<\gamma.
\]
More generally, if there exist \(\varepsilon_{\rm D}>0\) and a probability
distribution \(\nu\in\Delta_{|\mathcal S|}\) such that the Doeblin-type
minorization condition
\[
P(\cdot|s,a)\ge \varepsilon_{\rm D}\nu(\cdot),
\qquad
\forall (s,a)\in\mathcal S\times\mathcal A,
\]
holds componentwise, then every \(B_\pi\) has a common stochastic component of
mass at least \(\varepsilon_{\rm D}\). More precisely, for each fixed
deterministic policy \(\pi\), each row of \(B_\pi\) dominates the same
probability measure on \(\mathcal Y\) that assigns mass \(\nu(s')\) to
\((s',\pi(s'))\). This is a standard uniform mixing condition for finite
Markov chains~\cite{seneta2006non}. Consequently,
\[
\bar\rho\le \gamma(1-\varepsilon_{\rm D})<\gamma.
\]
These one-step conditions are strong but easy to check. They are sufficient, not
necessary: \(\bar\rho<\gamma\) may still hold when row overlap appears only
after several steps, as captured by the bounded scrambling condition in
\cref{prop:uniform-scrambling-strict}.

The same characterization also shows what can prevent strictness. For a
fixed deterministic policy \(\pi\), define the induced state transition matrix
\[
P_\pi:=\Pi^\pi P\in\mathbb R^{|\mathcal S|\times|\mathcal S|}.
\]
The nonzero eigenvalues of \(B_\pi=P\Pi^\pi\) coincide with those of \(P_\pi\).
Consequently, if some deterministic policy \(\pi\) has $|\lambda_2(P_\pi)|=1$, then, by~\cref{lem:restricted-spectral-radius}, $\rho(\bar A_\pi)=\gamma$.
Since \(\bar A_\pi\) is one member of the full restricted family, it follows that $\bar\rho\ge \rho(\bar A_\pi)=\gamma$. Together with \cref{lem:restricted-jsr-upper-gamma}, this gives $\bar\rho=\gamma$.
Therefore, strict inequality for the full family is impossible if even one
deterministic policy induces a state Markov chain with a nontrivial
unit-modulus mode, where a nontrivial unit-modulus mode means an eigenvalue
\(\lambda\) of \(P_\pi\) with \(|\lambda|=1\), other than the Perron eigenvalue
associated with the invariant constant direction. Such a mode prevents
contraction on the quotient space obtained after removing
\(\operatorname{span}(\mathbf 1)\)~\cite{seneta2006non,levin2017markov}.

Standard Markov-chain theory shows that this obstruction occurs, for example,
when \(P_\pi\) has more than one closed communicating class, meaning a
communicating class with no transition to states outside the class, or when a
closed communicating class is periodic, meaning that the greatest common
divisor of its possible return times is larger than one~\cite{seneta2006non,levin2017markov}.
Therefore, a necessary structural requirement for the full inequality
\(\bar\rho<\gamma\) is that every deterministic stationary policy be aperiodic
and unichain, where unichain means that there is exactly one closed
communicating class, possibly together with transient states. This
single-policy requirement alone, however, is not sufficient for arbitrary
switching; the precise full-family condition is the bounded scrambling
condition in \cref{prop:uniform-scrambling-strict}.

This full-family condition should be distinguished from the
post-identification condition in \cref{thm:two-stage-convergence}. After the
iterate enters the POSS, only the optimal restricted family $\bar{\mathcal H}_*
=
\{\bar A_\pi:\pi\in\Theta^*\}$ remains active. Hence, one only needs bounded scrambling over optimal policies to
obtain \(\bar\rho_*<\gamma\). In the unique-optimal-policy case, this reduces
to the single-chain condition
\[
|\lambda_2(P_{\pi^*})|<1,
\]
or equivalently, in standard finite-state Markov-chain terminology, that the
optimal-policy chain has one aperiodic closed communicating class, possibly
with transient states. This recovers the spectral statement in
\cref{cor:unique-optimal-policy} and gives a graph-theoretic interpretation of
when the post-identification transverse decay is strictly faster than the
classical \(\gamma\)-rate.

\section{Examples}\label{app:examples}
This subsection illustrates the geometric picture developed above
using a small discounted MDP with state set $\mathcal S=\{1,2,3\}$, action set $\mathcal A=\{1,2\}$, and discount factor $\gamma=0.95$. The transition probability matrices and the expected one-step rewards are chosen as
\[
P_1=
\begin{bmatrix}
0.7 & 0.2 & 0.1\\
0.2 & 0.6 & 0.2\\
0.1 & 0.3 & 0.6
\end{bmatrix},
\qquad
P_2=
\begin{bmatrix}
0.2 & 0.5 & 0.3\\
0.4 & 0.3 & 0.3\\
0.3 & 0.3 & 0.4
\end{bmatrix},
\qquad
R=
\begin{bmatrix}
1.0 & 0.2\\
0.6 & 0.0\\
1.2 & 0.3
\end{bmatrix},
\]
respectively. Here $P_a=P(\cdot\mid\cdot,a)$ for $a\in\mathcal A=\{1,2\}$, and the $(s,a)$-entry of $R$ represents the expected reward at state $s\in\mathcal S=\{1,2,3\}$ under action $a\in\mathcal A$, i.e. $R(s,a)$. For this example, the unique optimal deterministic policy is
\[
\pi^*=\bigl(\pi^*(1),\pi^*(2),\pi^*(3)\bigr)=(1,1,1),
\]
and the corresponding optimal Q-function is
\[
Q^*=
\begin{bmatrix}
18.2229 & 17.3026\\
17.6194 & 17.2172\\
18.4947 & 17.5430
\end{bmatrix},
\]
where the $(s,a)$-entry represents the optimal Q-function value at state $s$ under action $a$, i.e. $Q^*(s,a)$.
The corresponding minimum optimality gap is
\[
\Delta
:=
\min_{s\in {\cal S}_{\mathrm{sep}}}\bar{\Delta}_s,
\qquad
\bar{\Delta}_s
:=
V^*(s)-\max_{a\notin \Phi^*(s)}Q^*(s,a)
\approx 0.4022.
\]
We consider the affine space $\mathcal X_1 = Q^*+\operatorname{span}(\mathbf{1})$,
and choose the tube radius $\delta = 0.4\Delta \approx 0.1609$, which satisfies \(\delta<\Delta/2\).
To compare the convergence to the optimal Q-function and the convergence to the affine set
\(\mathcal X_1\), we plot the normalized quantities
\[
\frac{\|Q_k-Q^*\|_\infty}{\|Q_0-Q^*\|_\infty}
\qquad\text{and}\qquad
\frac{\dist_2(Q_k,\mathcal X_1)}{\dist_2(Q_0,\mathcal X_1)}.
\]
As reference curves, we also include $\gamma^k$ and $(\gamma|\lambda_2|)^k$, where \(|\lambda_2|\) denotes the modulus of the second largest eigenvalue of the state-action transition matrix \(P\Pi_{\pi^*}\) associated with the optimal policy. For this example, $|\lambda_2| \approx 0.5618$ and $\gamma|\lambda_2| \approx 0.5337$.

To visualize the geometry in two dimensions, we consider the two-dimensional affine plane
\[
Q^*+\operatorname{span}\{\hat {\mathbf{1}},\hat d\},
\]
where
\[
\hat {\mathbf{1}}:=\frac{1}{\sqrt{6}}(1,1,1,1,1,1)^\top \in {\mathbb R}^{6},
\qquad
\hat d:=\frac{1}{\sqrt{2}}(1,0,0,-1,0,0)^\top \in {\mathbb R}^{6}.
\]
Here, \(\hat {\mathbf{1}}\) is the normalized all-ones direction, which is tangent to \(\mathcal X_1\),
while \(\hat d\) is a transverse direction that perturbs the action contrast at the first state.
For any \(Q\), write the coordinates of \(Q-Q^*\) on this plane as
\[
u(Q):=\langle Q-Q^*,\hat {\mathbf{1}}\rangle,
\qquad
v(Q):=\langle Q-Q^*,\hat d\rangle.
\]
For the single-trajectory experiment, the initial condition is chosen on this tilted plane as follows:
\[
Q_0=
\begin{bmatrix}
19.5495 & 16.6292\\
17.9460 & 17.5438\\
18.8213 & 17.8696
\end{bmatrix}.
\]
For this particular \(Q_0\), the tie-broken greedy action is already the optimal action at every state. Thus, this single trajectory is used to visualize finite-time entrance into the invariant tube, rather than to demonstrate first-time policy identification from a non-optimal greedy policy. Starting from this \(Q_0\), we run Q-VI for \(50\) iterations. \Cref{fig:vi-normalized} shows the normalized decay of \(\|Q_k-Q^*\|_\infty\) and \(\dist_2(Q_k,\mathcal X_1)\).
The plot shows that the distance to \(\mathcal X_1\) decreases more rapidly than the full
Q-function error in the transient regime, in line with the theoretical picture
developed in this paper.
\Cref{fig:vi-projection-single} shows the orthogonal projection of the full Q-VI
trajectory onto the tilted plane. The dashed line represents the slice of \(\mathcal X_1\) and the shaded strip is the
slice of \({\cal T}_\delta\). The projected trajectory starts from the initial point, enters the strip in finite time, and then converges to the origin, which corresponds to \(Q_k\to Q^*\).

To show the global geometry more clearly, we also consider \(12\) initial conditions placed
uniformly on the circle of radius \(2\) in the tilted plane.
For each such initial condition, we run Q-VI for \(50\) iterations and project the
resulting trajectory onto the same tilted plane.
\Cref{fig:vi-projection-multiple} displays these projected trajectories together with the slice
of \(\mathcal X_1\), the slice of \({\cal T}_\delta\), and the initial circle.
This figure provides a global view of trajectories from different directions being
rapidly drawn toward the affine set \(\mathcal X_1\) before ultimately converging to \(Q^*\).

\begin{figure}[t]
\centering
\includegraphics[width=0.86\linewidth]{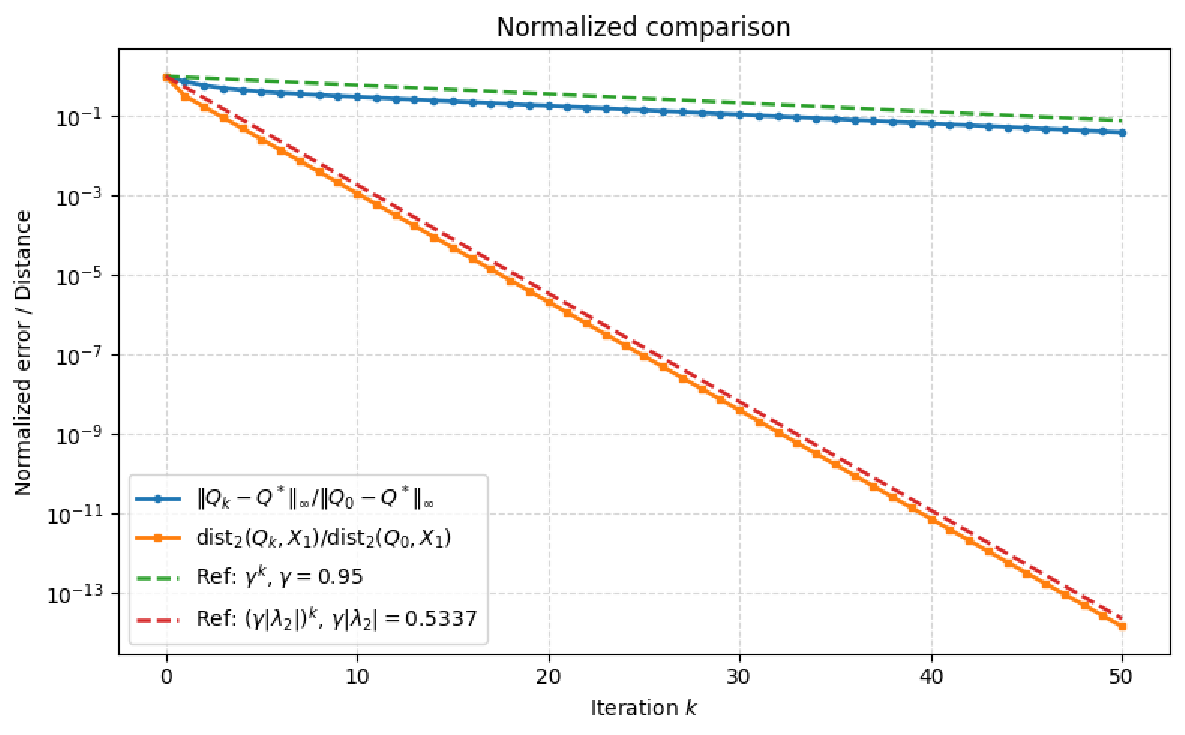}
\caption{Normalized comparison of
\(
\|Q_k-Q^*\|_\infty
\)
and
\(
\dist_2(Q_k,\mathcal X_1)
\)
for the toy MDP. The dashed curves indicate the reference rates
\(
\gamma^k
\)
and
\(
(\gamma|\lambda_2|)^k
\).
The plot shows that the iterate approaches the affine set \(\mathcal X_1\) faster than it
approaches the optimal Q-function \(Q^*\).}
\label{fig:vi-normalized}
\end{figure}
\begin{figure}[t]
\centering
\includegraphics[width=0.72\linewidth]{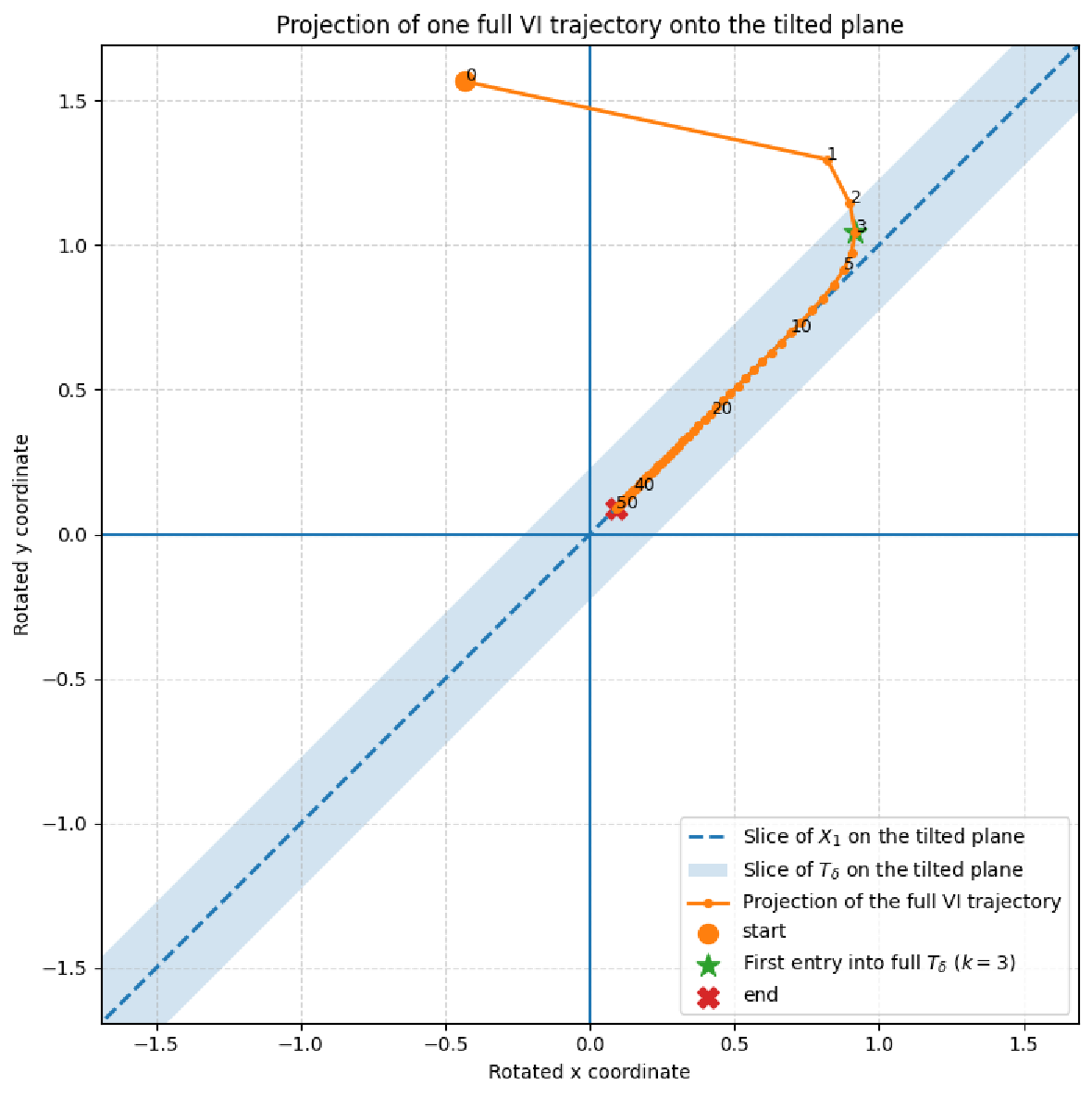}
\caption{Orthogonal projection of a single full Q-VI trajectory onto the tilted plane.
The dashed line is the slice of \(\mathcal X_1\) and the shaded strip is the slice of
the tube \({\cal T}_\delta\).
The projected trajectory starts from the initial point, enters the strip, and then converges to the origin.}
\label{fig:vi-projection-single}
\end{figure}
\begin{figure}[t]
\centering
\includegraphics[width=0.72\linewidth]{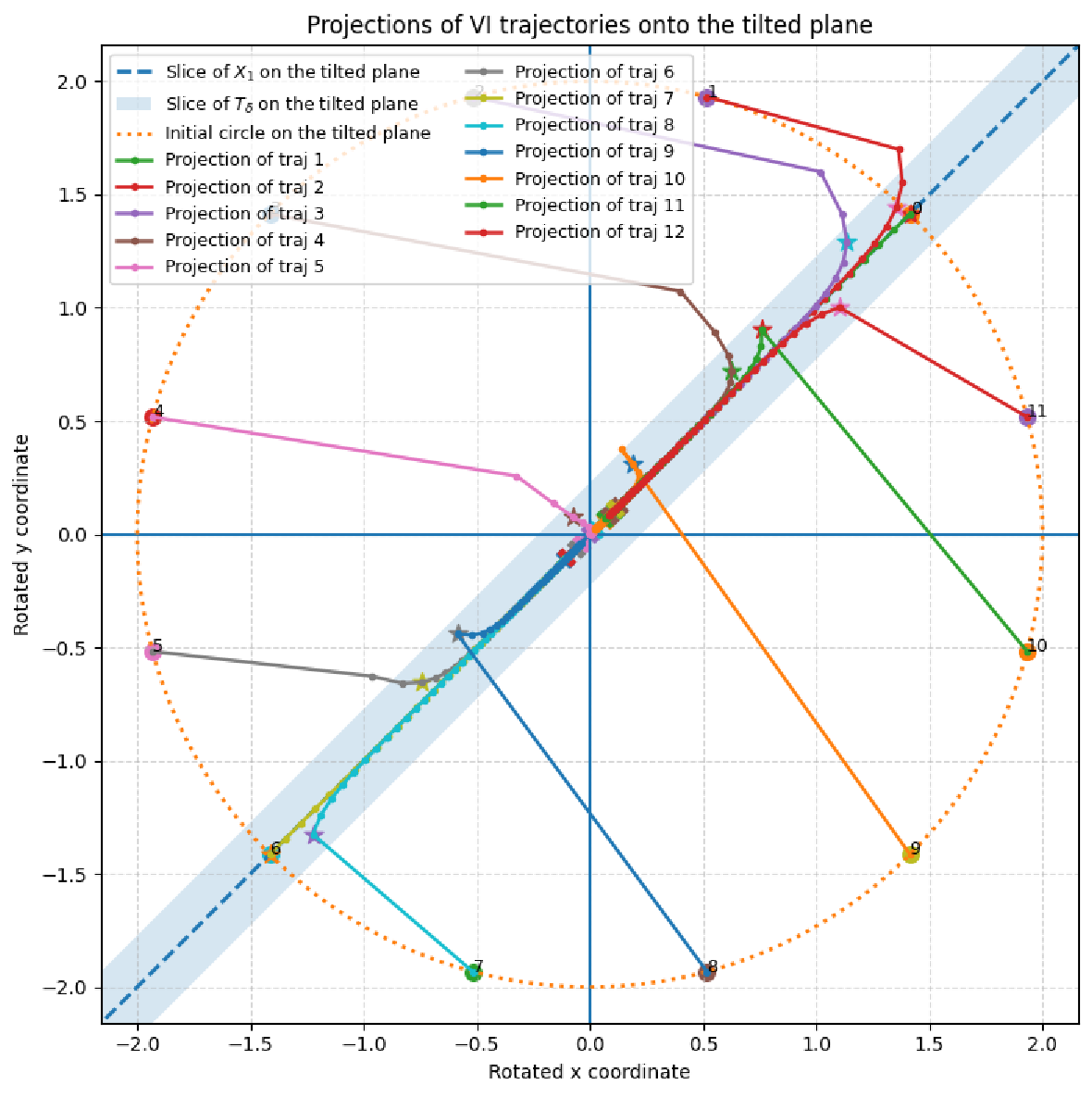}
\caption{Orthogonal projections of \(12\) full Q-VI trajectories onto the tilted plane,
where the initial conditions are chosen uniformly on a circle of radius \(2\) in the plane.
The dashed line is the slice of \(\mathcal X_1\), the shaded strip is the slice of \({\cal T}_\delta\), and the
dotted curve is the initial circle.
The plot gives a global geometric view of trajectories being first attracted toward
\(\mathcal X_1\) and then converging to \(Q^*\).}
\label{fig:vi-projection-multiple}
\end{figure}

We next present a tabular Q-learning example on the same toy MDP used in the preceding
Q-VI example. Since the underlying discounted MDP with $\mathcal S=\{1,2,3\}$ and $\mathcal A=\{1,2\}$, the optimal policy $\pi^*=\bigl(\pi^*(1),\pi^*(2),\pi^*(3)\bigr)=(1,1,1)$,
the optimal Q-function \(Q^*\), the affine set \(\mathcal X_1\), and the tube radius \(\delta\) are
the same as those introduced in the previous example, we omit their repeated description here.
We consider the standard asynchronous tabular Q-learning update. In the simulation, the action is sampled from the uniform behavior policy $\mu(a\mid s)=\frac{1}{|{\cal A}|}$ for $s\in {\cal S}$ and $a\in {\cal A}$, and the step size is chosen as $\alpha_t=\frac{0.35}{1+0.01\,t}$. To visualize the geometry, we consider the two-dimensional affine plane identical to the previous Q-VI case.
As before, we choose \(12\) initial conditions uniformly on the boundary of a circle
of radius \(r=2\) in the tilted plane.

\Cref{fig:ql-projection} shows the projected sample paths
corresponding to these \(12\) initial conditions, together with the line \(q=p\),
which is the projection of \(\mathcal X_1\), the strip corresponding to the projected tube
\({\cal T}_\delta\), and the boundary of the initial circle.
Since Q-learning is stochastic, the trajectories are not monotone and exhibit
path-dependent fluctuations. Nevertheless, a clear common trend can be observed:
from a variety of initial directions, the trajectories are first drawn toward the
line \(q=p\), enter the projected tube, and then continue to move toward the origin.
These paths are consistent with the geometric interpretation from the Q-VI example:
even in the stochastic tabular Q-learning setting, the affine set \(\mathcal X_1\) acts as a
transient attracting set before the iterates refine their convergence toward \(Q^*\).

Compared with the Q-VI example, the present figure displays a family of
stochastic sample paths rather than a single deterministic trajectory. This makes the
same geometric mechanism easier to see: although the noise causes trajectory-dependent oscillations, the iterates tend to approach the neighborhood of \(\mathcal X_1\)
relatively quickly and only afterward progress toward the optimal Q-function \(Q^*\).

\begin{figure}[t]
\centering
\includegraphics[width=0.82\linewidth]{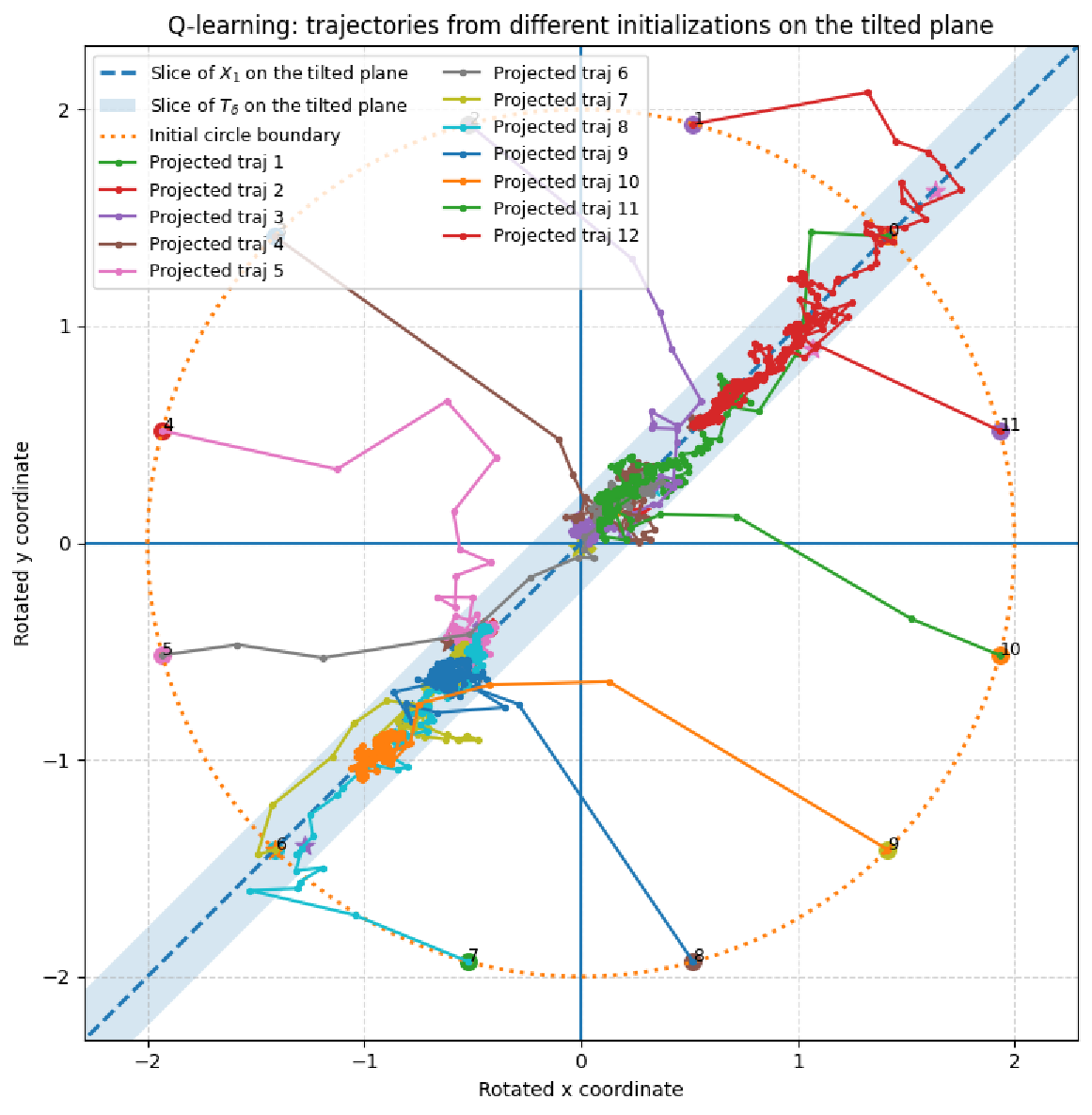}
\caption{Two-dimensional projection of tabular Q-learning trajectories onto the tilted plane
\(Q^*+\operatorname{span}\{\hat{\mathbf 1},\hat d\}\), displayed in the rotated coordinates
\(p=(u-v)/\sqrt{2}\) and \(q=(u+v)/\sqrt{2}\).
The dashed line \(q=p\) is the projection of \(\mathcal X_1\), and the shaded strip is the projection
of the tube \({\cal T}_\delta\), namely \( |q-p|\le 2\delta \).
The dotted circle indicates the boundary of the common initial set in the \((u,v)\)-plane.
Starting from \(12\) different initial points on this circle, the stochastic Q-learning
trajectories tend to enter the strip and then move toward the origin, which corresponds to
\(Q^*\).}
\label{fig:ql-projection}
\end{figure}

\section{Auxiliary statements and proofs}\label{app:moved-main-results}

\subsection{Statement and proof of Lemma~\ref{lem:X1-subset-Xstar}}
\label{app:proof-lem:X1-subset-Xstar}
\setcounter{lemma}{3}
\begin{lemma}\label{lem:X1-subset-Xstar}
We have $\mathcal X_1 \subset \mathcal X^*$.
\end{lemma}
\begin{proof}
Let $Q\in \mathcal X_1$. Then $Q=Q^*+\alpha \mathbf{1}$ for some $\alpha\in\mathbb R$. Since adding a constant
multiple of $\mathbf{1}$ does not change the action ordering,
\[
\Argmax_{a\in {\cal A}}Q(s,a)=\Argmax_{a\in {\cal A}}Q^*(s,a)=\Phi^*(s),\qquad \forall s\in {\cal S}.
\]
Hence every tie-broken greedy action of $Q$ belongs to $\Phi^*(s)$, so $Q\in \mathcal X^*$.
This completes the proof.
\end{proof}

\subsection{Statement and proof of Proposition~\ref{prop:X1-invariant}}
\label{app:proof-prop:X1-invariant}
\setcounter{proposition}{0}
\begin{proposition}[Invariance of $\mathcal X_1$]\label{prop:X1-invariant}
If $Q_k\in \mathcal X_1$, then $Q_{k+1}\in \mathcal X_1$. More precisely, if $Q_k=Q^*+\alpha_k \mathbf{1}$ for some $\alpha_k\in\mathbb R$, then $Q_{k+1}=Q^*+\gamma\alpha_k \mathbf{1}$.
\end{proposition}
\begin{proof}
Suppose that $Q_k=Q^*+\alpha_k \mathbf{1}$. Since adding a constant multiple of $\mathbf{1}$ does not change the
action ordering, every tie-broken greedy action of $Q_k$ is optimal for $Q^*$. Hence $\pi_{Q_k}\in\Theta^*$. Using the Bellman update, $Q_{k+1} = R+\gamma P\Pi_{Q_k}(Q^* +\alpha_k \mathbf{1})$. Because $\pi_{Q_k}$ is an optimal policy, we have $R+\gamma P\Pi_{Q_k}Q^*=Q^*$.
In addition, since $P\Pi_{Q_k}$ is stochastic, it satisfies $P\Pi_{Q_k} \mathbf{1}=\mathbf{1}$.
Therefore, it follows that $Q_{k+1}=Q^*+\gamma\alpha_k \mathbf{1}$, and $Q_{k+1}\in \mathcal X_1$, which proves the claim.
\end{proof}

\subsection{Proof of Proposition~\ref{prop:invariant-tube}}
\label{app:proof-prop:invariant-tube}
\begin{proof}
Let $Q\in {\cal T}_\delta$. By definition, there exists $\alpha\in\mathbb R$ such that $\|Q-(Q^*+\alpha\mathbf{1})\|_\infty\le \delta$. Fix any state $s\in {\cal S}_{\mathrm{sep}}$ and any action $b\notin\Phi^*(s)$. Then, we can derive the following inequalities:
\begin{align*}
{\max _{a \in {\Phi ^*}(s)}}Q(s,a) - Q(s,b) =& {\max _{a \in {\cal A}}}{Q^*}(s,a) - {Q^*}(s,b)\\
&+ \left( {{{\max }_{a \in {\Phi ^*}(s)}}Q(s,a) - \left( {{{\max }_{a \in {\cal A}}}{Q^*}(s,a) + \alpha } \right)} \right)\\
& - \left( {Q(s,b) - \left( {{Q^*}(s,b) + \alpha } \right)} \right)\\
\ge& {\max _{a \in {\cal A}}}{Q^*}(s,a) - {Q^*}(s,b)\\
&+ \left( {{{\max }_{a \in {\Phi ^*}(s)}}Q(s,a) - \left( {{{\max }_{a \in {\Phi ^*}(s)}}{Q^*}(s,a) + \alpha } \right)} \right)\\
&- \left( {Q(s,b) - \left( {{Q^*}(s,b) + \alpha } \right)} \right)\\
\ge& {{\bar \Delta }_s} - 2\delta \\
\ge& \bar \Delta  - 2\delta\\
 >& 0.
\end{align*}
Therefore no non-optimal action can be tie-broken greedy at any state $s\in {\cal S}_{\mathrm{sep}}$. For states outside
${\cal S}_{\mathrm{sep}}$, all actions are optimal by definition. Therefore, $\pi_Q(s)\in\Phi^*(s)$ for all $s\in {\cal S}$, and $Q\in \mathcal X^*$. Now let $Y:=Q^*+\alpha\mathbf{1}\in \mathcal X_1$ with $\|Q-Y\|_\infty\le \delta$. Since $Q\in \mathcal X^*$, we have $\pi_Q\in\Theta^*$. Although $\pi_Q$ need not be the tie-broken greedy policy of $Y$, the optimality of $\pi_Q$ implies
\[
R+\gamma P\Pi_Q Y
=
R+\gamma P\Pi_Q(Q^*+\alpha\mathbf{1})
=
Q^*+\gamma\alpha\mathbf{1}
\in \mathcal X_1.
\]
Moreover,
\[
F(Q) -\bigl(R+\gamma P\Pi_QY\bigr)
=
\gamma P\Pi_Q(Q-Y).
\]
Taking the infinity norm and using $\|P\Pi_Q\|_\infty=1$, we obtain
\[
\dist_\infty(F(Q),\mathcal X_1)
\le
\left\|F(Q)-\bigl(R+\gamma P\Pi_QY\bigr)\right\|_\infty
\le
\gamma\|Q-Y\|_\infty
\le
\gamma\delta < \delta.
\]
Therefore, one concludes that $F({\cal T}_\delta)\subset {\cal T}_{\gamma\delta}\subset {\cal T}_\delta$. This completes the proof.
\end{proof}

\subsection{Statement and proof of Corollary~\ref{cor:basic-entrance-Xstar}}
\label{app:proof-cor:basic-entrance-Xstar}
\setcounter{corollary}{0}
\begin{corollary}[Basic finite-time entrance into $\mathcal X^*$]\label{cor:basic-entrance-Xstar}
Let~\cref{assump:optimal-class-separation} hold. Define
\[
K_{\mathrm{basic}}
:=
\begin{cases}
0, & \|Q_0-Q^*\|_\infty<\dfrac{\bar\Delta}{2},\\[1.2ex]
\left\lfloor
\dfrac{
\log\!\left(\dfrac{2\|Q_0-Q^*\|_\infty}{\bar\Delta}\right)
}{
-\log\gamma
}
\right\rfloor+1,
& \|Q_0-Q^*\|_\infty\ge\dfrac{\bar\Delta}{2}.
\end{cases}
\]
Then, we have $Q_k\in \mathcal X^*$ for all $ k\ge K_{\mathrm{basic}}$.
\end{corollary}
\begin{proof}
By~\cref{eq:qvi-convergence}, one has $\|Q_k-Q^*\|_\infty \le \gamma^k\|Q_0-Q^*\|_\infty$.
The definition of $K_{\mathrm{basic}}$ guarantees
\[
\gamma^k\|Q_0-Q^*\|_\infty<\frac{\bar\Delta}{2},
\qquad \forall k\ge K_{\mathrm{basic}}.
\]
Hence, for all $k\ge K_{\mathrm{basic}}$, we get $\|Q_k-Q^*\|_\infty<\frac{\bar\Delta}{2}$.
Since $Q^*\in \mathcal X_1$, this implies $\dist_\infty(Q_k,\mathcal X_1)< \frac{\bar\Delta}{2}$.
For each such $k$, choose a number $\delta_k$ satisfying
\[
\dist_\infty(Q_k,\mathcal X_1)
\le
\delta_k
<
\frac{\bar\Delta}{2},
\]
which is possible because the distance is strictly smaller than \(\bar\Delta/2\). Then $Q_k\in{\cal T}_{\delta_k}$, and by~\cref{prop:invariant-tube} we have ${\cal T}_{\delta_k}\subset\mathcal X^*$. Hence $Q_k\in \mathcal X^*$ for all $k\ge K_{\mathrm{basic}}$, which completes the proof.
\end{proof}

\subsection{Proof of Lemma~\ref{lem:exact-stochastic-representation}}
\label{app:proof-lem:exact-stochastic-representation}
\begin{proof}
For each state $s\in {\cal S}$, define the optimality error
\[
\delta_k(s)
:=
\max_{a\in {\cal A}} Q_k(s,a) - \max_{a\in {\cal A}} Q^*(s,a).
\]
Since $Q_k(s,a)=Q^*(s,a)+e_k(s,a)$, we have
\[
\min_{a\in {\cal A}} e_k(s,a)
\le
\delta_k(s)
\le
\max_{a\in {\cal A}} e_k(s,a).
\]
Therefore, $\delta_k(s)$ belongs to the convex hull of the finite set
$\{e_k(s,a):a\in {\cal A}\}$. Hence there exists a probability vector $\mu_k(\cdot|s)\in \Delta_{|{\cal A}|}$ for each $k\geq 0$ such that
\[
{\delta _k}(s) = \sum\limits_{a \in {\cal A}} {{\mu _k}(a|s)} {e_k}(s,a) = {\mu _k}{( \cdot |s)^\top}{e_k}(s, \cdot ).
\]
Stacking these equalities over all states and using the definition of $\Pi^{\mu_k}$ in~\cref{eq:policy-transition-matrix}, with the same action-major ordering as the vector $e_k$, yields
\[
{\delta _k} =
\left[
\begin{array}{c}
{{\delta _k}(1)}\\
{{\delta _k}(2)}\\
\vdots\\
{{\delta _k}(|{\cal S}|)}
\end{array}
\right]
=
\left[
\begin{array}{c}
{\mu _k( \cdot \mid 1)^\top \otimes e_1^\top}\\
{\mu _k( \cdot \mid 2)^\top \otimes e_2^\top}\\
\vdots\\
{\mu _k( \cdot \mid |{\cal S}|)^\top \otimes e_{|{\cal S}|}^\top}
\end{array}
\right]
\left[
\begin{array}{c}
{{e_k}(\cdot,1)}\\
{{e_k}(\cdot,2)}\\
\vdots\\
{{e_k}(\cdot,|{\cal A}|)}
\end{array}
\right]
=
{\Pi ^{{\mu _k}}}{e_k} \in {{\mathbb R}^{|{\cal S}|}}.
\]
Now subtract the Bellman optimality equation from the Q-VI update:
\begin{align*}
e_{k+1}(s,a)
&=
\gamma \sum_{s'\in {\cal S}} P(s'|s,a)
\left(
\max_{b\in {\cal A}} Q_k(s',b) - \max_{b\in {\cal A}} Q^*(s',b)
\right) \\
&=
\gamma \sum_{s'\in {\cal S}} P(s'|s,a)\,\delta_k(s').
\end{align*}
In vector form, we have
\[
e_{k+1}
=
\gamma P\delta_k
=
\gamma P\Pi^{\mu_k}e_k
=
A_{\mu_k}e_k.
\]
This proves the claim.
\end{proof}

\subsection{Proof of Lemma~\ref{lem:projected-error-dynamics}}
\label{app:proof-lem:projected-error-dynamics}
\begin{proof}
Since \(z_k:=\Pi_\perp e_k\), we have the orthogonal decomposition
\[
e_k=(I-\Pi_\perp)e_k+\Pi_\perp e_k.
\]
Using $\Pi_\perp=I-\frac{1}{n}\mathbf{1}\mathbf{1}^\top$, it follows that
\[
(I-\Pi_\perp)e_k
=
\frac{1}{n}\mathbf{1}\mathbf{1}^\top e_k
=
\left(\frac{1}{n}\mathbf{1}^\top e_k\right)\mathbf{1}.
\]
Therefore, by setting $\alpha_k:=\frac{1}{n}\mathbf{1}^\top e_k$,
we may write
\[
e_k=\alpha_k\mathbf{1}+z_k.
\]
Then, we have
\[
A_{\mu_k}e_k
=
\gamma \alpha_k \mathbf{1} + A_{\mu_k}z_k.
\]
Applying $\proj$ to both sides yields
\[
z_{k+1}
=
\proj e_{k+1}
=
\proj A_{\mu_k}e_k
=
\proj A_{\mu_k} z_k
=
\proj A_{\mu_k}\proj z_k
=
\bar A_{\mu_k} z_k,
\]
which proves the first claim. For the second claim, note that $\mathcal X_1 = Q^* + \mathrm{span}(\mathbf{1})$.
Therefore, we get
\[
\dist_2(Q_k,\mathcal X_1)
=
\dist_2(e_k,\mathrm{span}(\mathbf{1}))
=
\|\proj e_k\|_2
=
\|z_k\|_2,
\]
which completes the proof.
\end{proof}

\subsection{Statement and proof of Lemma~\ref{lem:restricted-matrix-properties}}
\label{app:proof-lem:restricted-matrix-properties}
\setcounter{lemma}{6}
\begin{lemma}\label{lem:restricted-matrix-properties}
The following statements hold:
\begin{enumerate}
\item Let $(v,\lambda)$ be an eigenvector-eigenvalue pair of $A_i$. Then
\[
\bar A_i (\proj v) = \lambda (\proj v).
\]
In particular, if $\proj v \neq 0$, then $(\proj v,\lambda)$ is an eigenvector-eigenvalue pair of $\bar A_i$.

\item $(\mathbf{1},0)$ is an eigenvector-eigenvalue pair of $\bar A_i$.
\end{enumerate}
\end{lemma}
\begin{proof}
Recall that
\[
\bar A_i = \proj A_i \proj,
\qquad
\proj = I-\frac1n \mathbf{1} \mathbf{1}^\top,
\qquad
A_i \mathbf{1} = \gamma \mathbf{1}.
\]
To prove the first statement, let $(v,\lambda)$ be an eigenvector-eigenvalue pair of $A_i$ so that $A_i v = \lambda v$. Since $\proj v = v - \frac{\mathbf{1}^\top v}{n}\mathbf{1}$, we have
\[
A_i(\proj v)
=
A_i v - \frac{\mathbf{1}^\top v}{n}A_i\mathbf{1}
=
\lambda v - \frac{\gamma\,\mathbf{1}^\top v}{n}\mathbf{1}.
\]
Applying $\proj$ to both sides yields
\begin{align*}
\bar A_i(\proj v)
&=
\proj A_i \proj v \\
&=
\proj\left(\lambda v - \frac{\gamma\,\mathbf{1}^\top v}{n}\mathbf{1}\right) \\
&=
\lambda \proj v
-
\frac{\gamma\,\mathbf{1}^\top v}{n}\proj \mathbf{1}.
\end{align*}
Since $\proj \mathbf{1} = 0$, it follows that $\bar A_i(\proj v)=\lambda \proj v$. Therefore, if $\proj v\neq 0$, then $(\proj v,\lambda)$ is an eigenvector-eigenvalue pair of $\bar A_i$.
For the second statement, because $\proj \mathbf{1} = 0$, we have $\bar A_i \mathbf{1}
=
\proj A_i \proj \mathbf{1}
=
\proj A_i 0
=
0$. Therefore, $(\mathbf{1},0)$ is an eigenvector-eigenvalue pair of $\bar A_i$. This completes the proof.
\end{proof}

\subsection{Statement and proof of Lemma~\ref{lem:projection-product-identity}}
\label{app:proof-lem:projection-product-identity}
\setcounter{lemma}{7}
\begin{lemma}\label{lem:projection-product-identity}
For any vector \(x \in \mathbb{R}^n\), any positive integer \(k\), and any switching sequence $\bar{\sigma}_k := (\sigma_1,\sigma_2,\ldots,\sigma_k)\in\{1,2,\ldots,M\}^k$, we have
\[
\proj A_{\sigma_k}\cdots A_{\sigma_2}A_{\sigma_1}x
=
\bar A_{\sigma_k}\cdots \bar A_{\sigma_2}\bar A_{\sigma_1}x.
\]
\end{lemma}
\begin{proof}
We prove the claim by induction on \(k \geq 0\).
First, consider the case \(k=1\). Since $x = \proj x + (I-\proj)x$, we obtain
\[
\proj A_{\sigma_1}x
=
\proj A_{\sigma_1}\proj x
+
\proj A_{\sigma_1}(I-\proj)x.
\]
Now note that \((I-\proj)x \in \mathrm{span}(\mathbf{1})\), and hence, there exists a scalar \(c\in\mathbb R\) such that
\[
(I-\proj)x = c\mathbf{1}.
\]
Since \(A_i\mathbf{1}=\gamma \mathbf{1}\) for every \(i\), it follows that
\[
\proj A_{\sigma_1}(I-\proj)x
=
c\,\proj A_{\sigma_1}\mathbf{1}
=
c\gamma \proj \mathbf{1}
=0.
\]
Therefore,
\[
\proj A_{\sigma_1}x
=
\proj A_{\sigma_1}\proj x
=
\bar A_{\sigma_1}x.
\]
Hence, the claim holds for \(k=1\). Assume now that the identity holds for some \(k\ge 1\), namely,
\[
\proj A_{\sigma_k}\cdots A_{\sigma_2}A_{\sigma_1}x
=
\bar A_{\sigma_k}\cdots \bar A_{\sigma_2}\bar A_{\sigma_1}x.
\]
We show that it also holds for \(k+1\). We write
\[
A_{\sigma_k}\cdots A_{\sigma_1}x
=
\proj A_{\sigma_k}\cdots A_{\sigma_1}x
+
(I-\proj)A_{\sigma_k}\cdots A_{\sigma_1}x.
\]
Applying \(\proj A_{\sigma_{k+1}}\) to both sides gives
\begin{align*}
\proj A_{\sigma_{k+1}}A_{\sigma_k}\cdots A_{\sigma_1}x
&=
\proj A_{\sigma_{k+1}}\proj A_{\sigma_k}\cdots A_{\sigma_1}x \\
&\quad
+
\proj A_{\sigma_{k+1}}(I-\proj)A_{\sigma_k}\cdots A_{\sigma_1}x.
\end{align*}
Again, since $(I-\proj)A_{\sigma_k}\cdots A_{\sigma_1}x \in \mathrm{span}(\mathbf{1})$, there exists some scalar \(d\in\mathbb R\) such that
\[
(I-\proj)A_{\sigma_k}\cdots A_{\sigma_1}x = d\mathbf{1}.
\]
Using \(A_i\mathbf{1}=\gamma \mathbf{1}\) and \(\proj\mathbf{1}=0\), we obtain
\[
\proj A_{\sigma_{k+1}}(I-\proj)A_{\sigma_k}\cdots A_{\sigma_1}x
=
d\,\proj A_{\sigma_{k+1}}\mathbf{1}
=
d\gamma \proj \mathbf{1}
=
0.
\]
Hence,
\[
\proj A_{\sigma_{k+1}}A_{\sigma_k}\cdots A_{\sigma_1}x
=
\proj A_{\sigma_{k+1}}\proj A_{\sigma_k}\cdots A_{\sigma_1}x
=
\bar A_{\sigma_{k+1}}\proj A_{\sigma_k}\cdots A_{\sigma_1}x.
\]
By the induction hypothesis, $\proj A_{\sigma_k}\cdots A_{\sigma_1}x
=
\bar A_{\sigma_k}\cdots \bar A_{\sigma_1}x$. Substituting this into the previous equality yields
\[
\proj A_{\sigma_{k+1}}A_{\sigma_k}\cdots A_{\sigma_1}x
=
\bar A_{\sigma_{k+1}}\bar A_{\sigma_k}\cdots \bar A_{\sigma_1}x.
\]
Therefore, the claim holds for \(k+1\). By induction, the result follows for all \(k\ge 1\).
\end{proof}

\subsection{Statement and proof of Lemma~\ref{lem:restricted-spectral-radius}}
\label{app:proof-lem:restricted-spectral-radius}
\setcounter{lemma}{8}
\begin{lemma}\label{lem:restricted-spectral-radius}
Let the eigenvalues of \(B_i:=\gamma^{-1}A_i=P\Pi^{\pi_i}\), counted with algebraic multiplicity over the complexification, be ordered as
\[
\lambda_{1,i}=1,\lambda_{2,i},\ldots,\lambda_{n,i},
\qquad
n:=|{\cal S}||{\cal A}|,
\]
where \(\lambda_{1,i}=1\) is the Perron eigenvalue associated with the invariant constant direction and the remaining eigenvalues are ordered so that
\[
|\lambda_{2,i}|\ge |\lambda_{3,i}|\ge \cdots \ge |\lambda_{n,i}|.
\]
If the eigenvalue \(1\) has algebraic multiplicity larger than one, another copy of \(1\) may appear among \(\lambda_{2,i},\ldots,\lambda_{n,i}\). Then,
\[
\rho(\bar A_i)=\gamma |\lambda_{2,i}|,
\]
where \(\rho\) denotes the spectral radius.
\end{lemma}
\begin{proof}
All spectral statements below are understood over \(\mathbb C\). Recall that \(A_i=\gamma P\Pi^{\pi_i}\) satisfies $A_i \mathbf{1} = \gamma \mathbf{1}$.
Let \(W:=\mathrm{span}(\mathbf{1})^\perp\), and choose real vectors $w_2,\ldots,w_n$ forming a basis of \(W\). Define the invertible matrix
\[
T:=\begin{bmatrix}\mathbf{1} & w_2 & \cdots & w_n\end{bmatrix}\in\mathbb R^{n\times n}.
\]
Since \(A_i\mathbf{1}=\gamma \mathbf{1}\), the one-dimensional subspace \(\mathrm{span}(\mathbf{1})\) is \(A_i\)-invariant. Therefore, in the basis induced by \(T\), the matrix \(A_i\) has the block upper-triangular form
\[
T^{-1}A_iT
=
\begin{bmatrix}
\gamma & \alpha_i^\top\\
0 & \Gamma_i
\end{bmatrix}
\]
for some vector \(\alpha_i\in\mathbb R^{n-1}\) and some matrix \(\Gamma_i\in\mathbb R^{(n-1)\times(n-1)}\). Complexifying this block representation does not change the eigenvalues.
Next, since \(\proj\) is the orthogonal projection onto \(W=\mathrm{span}(\mathbf{1})^\perp\), we have
\[
\proj \mathbf{1} =0,
\qquad
\proj w_j = w_j,\quad j=2,\ldots,n.
\]
Hence, in the same basis,
\[
T^{-1}\proj T
=
\begin{bmatrix}
0 & 0\\
0 & I_{n-1}
\end{bmatrix}.
\]
Therefore,
\begin{align*}
T^{-1}\bar A_i T
&=T^{-1}\proj A_i \proj T\\
&=(T^{-1}\proj T)(T^{-1}A_iT)(T^{-1}\proj T)\\
&=
\begin{bmatrix}
0 & 0\\
0 & I_{n-1}
\end{bmatrix}
\begin{bmatrix}
\gamma & \alpha_i^\top\\
0 & \Gamma_i
\end{bmatrix}
\begin{bmatrix}
0 & 0\\
0 & I_{n-1}
\end{bmatrix}\\
&=
\begin{bmatrix}
0 & 0\\
0 & \Gamma_i
\end{bmatrix}.
\end{align*}
It follows that
\[
\sigma(\bar A_i)=\{0\}\cup \sigma(\Gamma_i).
\]
On the other hand, since
\[
T^{-1}A_iT
=
\begin{bmatrix}
\gamma & \alpha_i^\top\\
0 & \Gamma_i
\end{bmatrix}
\]
is block upper triangular, its spectrum is
\[
\sigma(A_i)=\{\gamma\}\cup \sigma(\Gamma_i).
\]
Equivalently, if the eigenvalues of \(B_i=\gamma^{-1}A_i\) are ordered as in the statement, then the eigenvalues of \(A_i\) are $\gamma,\gamma\lambda_{2,i},\ldots,\gamma\lambda_{n,i}$, and the eigenvalues of \(\bar A_i\) are $0,\gamma\lambda_{2,i},\ldots,\gamma\lambda_{n,i}$. Therefore,
\[
\rho(\bar A_i)
=
\max\bigl\{0,\,\gamma|\lambda_{2,i}|,\ldots,\gamma|\lambda_{n,i}|\bigr\}
=
\gamma |\lambda_{2,i}|.
\]
This completes the proof.
\end{proof}

\subsection{Statement and proof of Lemma~\ref{lem:jsr-upper-bound}}
\label{app:proof-lem:jsr-upper-bound}
\setcounter{lemma}{9}
\begin{lemma}[A computable upper bound on the JSR]\label{lem:jsr-upper-bound}
Let \(\bar{\cal H}:=\{\bar A_1,\bar A_2,\ldots,\bar A_M\}\) and suppose that there exist a submultiplicative matrix norm \(\|\cdot\|_*\) and a constant \(\bar\beta\in(0,1)\) such that
\[
\|\bar A_i\|_* \le \bar\beta,
\qquad
\forall i\in\{1,2,\ldots,M\}.
\]
Then the JSR of \(\bar{\cal H}\) satisfies $\rho(\bar A_1,\bar A_2,\ldots,\bar A_M)\le \bar\beta$.
\end{lemma}
\begin{proof}
By the definition of the JSR, for any positive integer \(k\),
\[
\rho(\bar A_1,\ldots,\bar A_M)
=
\lim_{k\to\infty}
\max_{\bar\sigma_k\in\{1,\ldots,M\}^k}
\left\|
\bar A_{\sigma_k}\cdots \bar A_{\sigma_1}
\right\|_*^{1/k}.
\]
Fix any switching sequence $\bar\sigma_k=(\sigma_1,\sigma_2,\ldots,\sigma_k)\in\{1,\ldots,M\}^k$.
Since \(\|\cdot\|_*\) is submultiplicative, we have
\[
\left\|
\bar A_{\sigma_k}\cdots \bar A_{\sigma_1}
\right\|_*
\le
\prod_{j=1}^k \|\bar A_{\sigma_j}\|_*
\le
\bar\beta^k.
\]
Taking the maximum over all switching sequences yields $\max_{\bar\sigma_k\in\{1,\ldots,M\}^k}
\left\|
\bar A_{\sigma_k}\cdots \bar A_{\sigma_1}
\right\|_*
\le
\bar\beta^k$. Taking the \(k\)-th root and letting \(k\to\infty\), we obtain $\rho(\bar A_1,\bar A_2,\ldots,\bar A_M)\le \bar\beta$.
This completes the proof.
\end{proof}

\subsection{Proof of Lemma~\ref{lem:restricted-jsr-upper-gamma}}
\label{app:proof-lem:restricted-jsr-upper-gamma}
\begin{proof}
Fix any submultiplicative matrix norm \(\|\cdot\|\), any positive integer \(k\geq 0\), and any switching sequence $\bar\sigma_k=(\sigma_1,\ldots,\sigma_k)\in\{1,2,\ldots,M\}^k$.
\cref{lem:projection-product-identity} yields
\[
\bar A_{\sigma_k}\cdots \bar A_{\sigma_1}x
=
\proj A_{\sigma_k}\cdots A_{\sigma_1}x
\qquad
\forall x\in\mathbb R^n.
\]
Hence, this implies $\bar A_{\sigma_k}\cdots \bar A_{\sigma_1}
= \proj A_{\sigma_k}\cdots A_{\sigma_1}$.
Therefore, it follows that
\[
\bigl\|\bar A_{\sigma_k}\cdots \bar A_{\sigma_1}\bigr\|
\le
\|\proj\|\,\bigl\|A_{\sigma_k}\cdots A_{\sigma_1}\bigr\|.
\]
Taking the maximum over all switching sequences gives
\[
\max_{\bar\sigma_k}
\bigl\|\bar A_{\sigma_k}\cdots \bar A_{\sigma_1}\bigr\|
\le
\|\proj\|
\max_{\bar\sigma_k}
\bigl\|A_{\sigma_k}\cdots A_{\sigma_1}\bigr\|.
\]
Taking the \(k\)-th root yields
\[
\left(
\max_{\bar\sigma_k}
\bigl\|\bar A_{\sigma_k}\cdots \bar A_{\sigma_1}\bigr\|
\right)^{1/k}
\le
\|\proj\|^{1/k}
\left(
\max_{\bar\sigma_k}
\bigl\|A_{\sigma_k}\cdots A_{\sigma_1}\bigr\|
\right)^{1/k}.
\]
Letting \(k\to\infty\), and using \(\|\proj\|^{1/k}\to 1\), we obtain
\[
\rho(\bar A_1,\ldots,\bar A_M)
\le
\rho(A_1,\ldots,A_M).
\]
Since \(\rho(A_1,\ldots,A_M)=\gamma\) from~\cref{lem:full-jsr-gamma}, it follows that $\rho(\bar A_1,\ldots,\bar A_M)\le \gamma$. This completes the proof.
\end{proof}

\subsection{Statement and proof of Lemma~\ref{lem:common-lyapunov-restricted-family}}
\label{app:proof-lem:common-lyapunov-restricted-family}
\setcounter{lemma}{11}
\begin{lemma}[Common Lyapunov function for the restricted switching family]\label{lem:common-lyapunov-restricted-family}
Let
\[
\bar{\mathcal H}:=\{\bar A_1,\bar A_2,\ldots,\bar A_M\},
\qquad
\bar \rho := \rho(\bar A_1,\bar A_2,\ldots,\bar A_M),
\]
and fix any $\epsilon>0$ such that
\[
\beta_\epsilon:=\bar\rho+\epsilon \in (0,1).
\]
For each integer \(t\ge 0\), define the function
\[
V_\varepsilon^t(x)
:=
\sum_{k=0}^t
\beta_\varepsilon^{-2k}
\max_{\bar \sigma_k\in\{1,2,\ldots,M\}^k}
\left\|
\bar A_{\sigma_k}\cdots \bar A_{\sigma_1}x
\right\|_2^2,
\qquad x\in\mathbb R^n.
\]
Then the following statements hold:
\begin{enumerate}
\item For every \(t\ge 0\),
\[
V_\varepsilon^{t+1}(x)
\ge
\|x\|_2^2
+
\beta_\varepsilon^{-2}
\max_{i\in\{1,\ldots,M\}}
V_\varepsilon^t(\bar A_i x).
\]

\item For every \(t\ge 0\) and every \(x\in\mathbb R^n\), $V_\varepsilon^t(\lambda x)=|\lambda|^2 V_\varepsilon^t(x)$ for all $\lambda\in\mathbb R$, and $V_\varepsilon^t(x)\le V_\varepsilon^{t+1}(x)$.

\item There exists a constant \(C_\varepsilon>0\) such that
\[
\|x\|_2^2
\le
V_\varepsilon^t(x)
\le
C_\varepsilon \|x\|_2^2,
\qquad
\forall x\in\mathbb R^n,\ \forall t\ge 0.
\]

\item For every \(x\in\mathbb R^n\), the limit
\[
V_\varepsilon^\infty(x):=\lim_{t\to\infty}V_\varepsilon^t(x)
\]
exists and is finite. Moreover, $\|x\|_2^2
\le
V_\varepsilon^\infty(x)
\le
C_\varepsilon \|x\|_2^2$.

\item The function $p_\varepsilon(x):=\sqrt{V_\varepsilon^\infty(x)}$ is a norm on \(\mathbb R^n\).

\item The function \(V_\varepsilon^\infty\) satisfies the Lyapunov inequality
\[
V_\varepsilon^\infty(\bar A_i x)
\le
\beta_\varepsilon^2 V_\varepsilon^\infty(x),
\qquad
\forall x\in\mathbb R^n,\ \forall i\in\{1,\ldots,M\}.
\]
Equivalently,
\[
p_\varepsilon(\bar A_i x)
\le
\beta_\varepsilon p_\varepsilon(x),
\qquad
\forall x\in\mathbb R^n,\ \forall i\in\{1,\ldots,M\}.
\]

\item Consequently,
\[
\|\bar A_i\|_{p_\varepsilon}\le \beta_\varepsilon,
\qquad \forall i\in\{1,\ldots,M\},
\]
where $\|\bar A_i\|_{p_\varepsilon}$ is the induced matrix norm generated by \(p_\varepsilon\)
\[
\|\bar A_i\|_{p_\varepsilon}
:=
\sup_{x\neq 0}\frac{p_\varepsilon(\bar A_i x)}{p_\varepsilon(x)}.
\]
Therefore, we have $\rho(\bar A_1,\bar A_2,\ldots,\bar A_M)\le \beta_\varepsilon$.
\end{enumerate}
\end{lemma}
\begin{proof}
We prove the statements one by one.

\medskip
\noindent
\textit{Proof of 1).}
By definition, one has
\[
V_\varepsilon^{t+1}(x)
=
\sum_{k=0}^{t+1}
\beta_\varepsilon^{-2k}
\max_{\bar \sigma_k\in\{1,\ldots,M\}^k}
\left\|
\bar A_{\sigma_k}\cdots \bar A_{\sigma_1}x
\right\|_2^2.
\]
The \(k=0\) term is simply \(\|x\|_2^2\). For \(k\ge 1\), writing \(k=j+1\), we obtain
\begin{align*}
V_\varepsilon^{t+1}(x)
&=
\|x\|_2^2
+
\sum_{j=0}^{t}
\beta_\varepsilon^{-2(j+1)}
\max_{\bar \sigma_{j+1}}
\left\|
\bar A_{\sigma_{j+1}}\cdots \bar A_{\sigma_1}x
\right\|_2^2 \\
&=
\|x\|_2^2
+
\beta_\varepsilon^{-2}
\sum_{j=0}^{t}
\beta_\varepsilon^{-2j}
\max_{i\in\{1,\ldots,M\}}
\max_{\bar \tau_j\in\{1,\ldots,M\}^j}
\left\|
\bar A_{\tau_j}\cdots \bar A_{\tau_1}\bar A_i x
\right\|_2^2 \\
&\ge
\|x\|_2^2
+
\beta_\varepsilon^{-2}
\max_{i\in\{1,\ldots,M\}}
\sum_{j=0}^{t}
\beta_\varepsilon^{-2j}
\max_{\bar \tau_j\in\{1,\ldots,M\}^j}
\left\|
\bar A_{\tau_j}\cdots \bar A_{\tau_1}\bar A_i x
\right\|_2^2 \\
&=
\|x\|_2^2
+
\beta_\varepsilon^{-2}
\max_{i\in\{1,\ldots,M\}}
V_\varepsilon^t(\bar A_i x).
\end{align*}
This proves 1).

\medskip
\noindent
\textit{Proof of 2).}
The homogeneity follows immediately from the homogeneity of the Euclidean norm:
\[
V_\varepsilon^t(\lambda x)
=
\sum_{k=0}^{t}
\beta_\varepsilon^{-2k}
\max_{\bar \sigma_k}
\|\bar A_{\sigma_k}\cdots \bar A_{\sigma_1}(\lambda x)\|_2^2
=
|\lambda|^2 V_\varepsilon^t(x).
\]
The monotonicity $V_\varepsilon^t(x)\le V_\varepsilon^{t+1}(x)$ holds because \(V_\varepsilon^{t+1}\) contains all terms of \(V_\varepsilon^t\) plus one additional nonnegative term.

\medskip
\noindent
\textit{Proof of 3).}
The lower bound is immediate from the \(k=0\) term: $V_\varepsilon^t(x)\ge \|x\|_2^2$.
For the upper bound, since \(\bar \rho<\beta_\varepsilon\), choose any number \(\eta\) such that $
\bar \rho<\eta<\beta_\varepsilon$. By the definition of the JSR, there exists an integer \(K \geq 0\) such that
\[
\max_{\bar \sigma_k\in\{1,\ldots,M\}^k}
\left\|
\bar A_{\sigma_k}\cdots \bar A_{\sigma_1}
\right\|_2^{1/k}
\le \eta,
\qquad \forall k\ge K.
\]
Hence, for all \(k\ge K\), we have
\[
\max_{\bar \sigma_k\in\{1,\ldots,M\}^k}
\left\|
\bar A_{\sigma_k}\cdots \bar A_{\sigma_1}
\right\|_2
\le
\eta^k.
\]
Now define
\[
C_0:=\max\Bigl\{1,\,
\max_{0\le k\le K-1}
\eta^{-k}
\max_{\bar \sigma_k}
\|\bar A_{\sigma_k}\cdots \bar A_{\sigma_1}\|_2
\Bigr\}.
\]
Then, for every \(k\ge 0\),
\[
\max_{\bar \sigma_k}
\left\|
\bar A_{\sigma_k}\cdots \bar A_{\sigma_1}
\right\|_2
\le
C_0\,\eta^k.
\]
Therefore, it follows that
\begin{align*}
V_\varepsilon^t(x)
&=
\sum_{k=0}^{t}
\beta_\varepsilon^{-2k}
\max_{\bar \sigma_k}
\left\|
\bar A_{\sigma_k}\cdots \bar A_{\sigma_1}x
\right\|_2^2 \\
&\le
\sum_{k=0}^{t}
\beta_\varepsilon^{-2k}
\left(
\max_{\bar \sigma_k}
\left\|
\bar A_{\sigma_k}\cdots \bar A_{\sigma_1}
\right\|_2
\right)^2
\|x\|_2^2 \\
&\le
C_0^2
\sum_{k=0}^{t}
\left(\frac{\eta}{\beta_\varepsilon}\right)^{2k}
\|x\|_2^2 \\
&\le
C_0^2
\sum_{k=0}^{\infty}
\left(\frac{\eta}{\beta_\varepsilon}\right)^{2k}
\|x\|_2^2.
\end{align*}
Since \(\eta/\beta_\varepsilon<1\), the geometric series converges. Thus, setting
\[
C_\varepsilon:=
\frac{C_0^2}{1-(\eta/\beta_\varepsilon)^2},
\]
we obtain $V_\varepsilon^t(x)\le C_\varepsilon \|x\|_2^2$ for all $x$ and $t\ge 0$. This proves 3).

\medskip
\noindent
\textit{Proof of 4).}
By 2), the sequence \(V_\varepsilon^t(x)\) is nondecreasing in \(t\), and by 3), it is uniformly bounded above by \(C_\varepsilon\|x\|_2^2\). Hence the limit $V_\varepsilon^\infty(x):=\lim_{t\to\infty}V_\varepsilon^t(x)$ exists and is finite for every \(x\in\mathbb R^n\). Passing to the limit in the bounds of 3) yields $\|x\|_2^2
\le
V_\varepsilon^\infty(x)
\le
C_\varepsilon\|x\|_2^2$.

\medskip
\noindent
\textit{Proof of 5).}
For each fixed \(k\ge 1\), define
\[
\nu_k(x)
:=
\beta_\varepsilon^{-k}
\max_{\bar \sigma_k\in\{1,\ldots,M\}^k}
\|\bar A_{\sigma_k}\cdots \bar A_{\sigma_1}x\|_2.
\]
Moreover, define $\nu_0(x):=\|x\|_2$.
For each \(k\ge 0\), \(\nu_k\) is a seminorm, since it is the pointwise maximum of seminorms. Then
\[
V_\varepsilon^t(x)=\sum_{k=0}^t \nu_k(x)^2,
\qquad
p_\varepsilon^t(x):=\sqrt{V_\varepsilon^t(x)}
=
\left(\sum_{k=0}^t \nu_k(x)^2\right)^{1/2}.
\]
Because \(\nu_0(x)=\|x\|_2\) is a norm, \(p_\varepsilon^t\) is a norm for every \(t\ge 0\). Indeed, positivity and absolute homogeneity are immediate, and the triangle inequality follows from Minkowski's inequality applied to the vector $(\nu_0(x),\nu_1(x),\ldots,\nu_t(x))$. Now, by definition,
\[
p_\varepsilon(x)=\sqrt{V_\varepsilon^\infty(x)}=\lim_{t\to\infty}p_\varepsilon^t(x),
\]
where the limit is monotone increasing. Since each \(p_\varepsilon^t\) is a norm, we have for all \(x,y\in\mathbb R^n\), $p_\varepsilon^t(x+y)\le p_\varepsilon^t(x)+p_\varepsilon^t(y)$. Letting \(t\to\infty\), we obtain
\[
p_\varepsilon(x+y)\le p_\varepsilon(x)+p_\varepsilon(y).
\]
Absolute homogeneity is inherited from \(V_\varepsilon^\infty(\lambda x)=|\lambda|^2V_\varepsilon^\infty(x)\), and positive definiteness follows from
\[
p_\varepsilon(x)^2=V_\varepsilon^\infty(x)\ge \|x\|_2^2.
\]
Hence \(p_\varepsilon\) is a norm.

\medskip
\noindent
\textit{Proof of 6).}
Using 1), we have
\[
V_\varepsilon^{t+1}(x)
\ge \|x\|_2^2
+\beta_\varepsilon^{-2}
\max_i V_\varepsilon^t(\bar A_i x).
\]
Therefore,
\[
\max_i V_\varepsilon^t(\bar A_i x)
\le
\beta_\varepsilon^2\bigl(V_\varepsilon^{t+1}(x)-\|x\|_2^2\bigr)
\le
\beta_\varepsilon^2 V_\varepsilon^{t+1}(x).
\]
Fixing \(i\), we have $V_\varepsilon^t(\bar A_i x)\le \beta_\varepsilon^2 V_\varepsilon^{t+1}(x)$. Letting \(t\to\infty\) and using the monotone convergence of both sides gives
\[
V_\varepsilon^\infty(\bar A_i x)
\le
\beta_\varepsilon^2 V_\varepsilon^\infty(x).
\]
Taking square roots yields $p_\varepsilon(\bar A_i x)\le \beta_\varepsilon p_\varepsilon(x)$.

\medskip
\noindent
\textit{Proof of 7).}
From 6), the induced matrix norm generated by \(p_\varepsilon\) satisfies
\[
\|\bar A_i\|_{p_\varepsilon}
:=
\sup_{x\neq 0}\frac{p_\varepsilon(\bar A_i x)}{p_\varepsilon(x)}
\le
\beta_\varepsilon,
\qquad \forall i\in\{1,\ldots,M\}.
\]
Hence, for any switching sequence \((\sigma_1,\ldots,\sigma_k)\),
\[
\|\bar A_{\sigma_k}\cdots \bar A_{\sigma_1}\|_{p_\varepsilon}
\le
\prod_{j=1}^k \|\bar A_{\sigma_j}\|_{p_\varepsilon}
\le
\beta_\varepsilon^k.
\]
Taking the maximum over all switching sequences, then the \(k\)-th root, and finally the limit as \(k\to\infty\), we obtain $\rho(\bar A_1,\bar A_2,\ldots,\bar A_M)\le \beta_\varepsilon$.
This completes the proof.
\end{proof}

\subsection{Statement and proof of Lemma~\ref{lem:convex-hull-extension}}
\label{app:proof-lem:convex-hull-extension}
\setcounter{lemma}{12}
\begin{lemma}[Convex-hull extension of the Lyapunov function]
\label{lem:convex-hull-extension}
Let $V_\infty^\varepsilon$ be the convex homogeneous Lyapunov function defined in \cref{lem:common-lyapunov-restricted-family},
and fix any $\epsilon>0$ such that
\[
\beta_\epsilon:=\bar\rho+\epsilon \in (0,1).
\]
Then $V_\infty^\varepsilon$ is convex, and for every stochastic policy $\mu$,
\[
V_\infty^\varepsilon(\bar A_\mu x)
\le
\beta_\varepsilon^2 V_\infty^\varepsilon(x),
\qquad \forall x\in\mathbb R^{|{\cal S}||{\cal A}|}.
\]
\end{lemma}
\begin{proof}
For each deterministic policy $\pi\in\Theta$, let
\[
A_\pi := \gamma P\Pi^\pi,
\qquad
\bar A_\pi := \proj A_\pi \proj.
\]
Any stochastic policy $\mu$ can be represented as a convex combination of deterministic
policies:
\[
\Pi^\mu
=
\sum_{\pi\in\Theta} c_\pi(\mu)\Pi^\pi,
\qquad
c_\pi(\mu)\ge 0,
\qquad
\sum_{\pi\in\Theta} c_\pi(\mu)=1,
\]
where
\[
c_\pi(\mu):=\prod_{s\in {\cal S}}\mu(\pi(s)|s).
\]
Consequently, we have
\[
A_\mu = \sum_{\pi\in\Theta} c_\pi(\mu) A_\pi,
\qquad
\bar A_\mu = \sum_{\pi\in\Theta} c_\pi(\mu)\bar A_\pi.
\]
Conversely, any convex combination of deterministic-policy matrices is generated by the stochastic policy whose statewise action probabilities are the corresponding marginals,
\[
\mu_c(a|s):=\sum_{\pi\in\Theta:\,\pi(s)=a}c_\pi.
\]
Thus the stochastic-policy family is exactly the convex hull of the deterministic-policy family; the proof below uses only the inclusion displayed above.
Now, for each finite $t$, the function $V_t^\varepsilon$ is convex because it is a sum of terms
of the form $x \mapsto \max_{\bar\sigma_k}\|\bar A_{\sigma_k}\cdots \bar A_{\sigma_1}x\|_2^2$, which is the pointwise maximum of convex quadratic functions. Since $V_\infty^\varepsilon(x)=\sup_{t\ge 0} V_t^\varepsilon(x)$, it follows that $V_\infty^\varepsilon$ is also convex.
Therefore, using Jensen's inequality leads to
\begin{align*}
V_\infty^\varepsilon(\bar A_\mu x)
&=
V_\infty^\varepsilon\!\left(
\sum_{\pi\in\Theta} c_\pi(\mu)\bar A_\pi x
\right) \\
&\le
\sum_{\pi\in\Theta} c_\pi(\mu)\,
V_\infty^\varepsilon(\bar A_\pi x) \\
&\le
\max_{\pi\in\Theta} V_\infty^\varepsilon(\bar A_\pi x).
\end{align*}
By~\cref{lem:common-lyapunov-restricted-family}, one gets $\max_{\pi\in\Theta} V_\infty^\varepsilon(\bar A_\pi x) \le \beta_\varepsilon^2 V_\infty^\varepsilon(x)$. Combining the two inequalities completes the proof.
\end{proof}

\subsection{Proof of Theorem~\ref{thm:global-exp-conv-X1}}
\label{app:proof-thm:global-exp-conv-X1}
\begin{proof}
By~\cref{lem:projected-error-dynamics}, we have $z_{k+1}=\bar A_{\mu_k} z_k$.
Hence, by~\cref{lem:convex-hull-extension}, one can derive
\[
V_\infty^\varepsilon(z_{k+1})
\le
\beta_\varepsilon^2 V_\infty^\varepsilon(z_k),
\qquad \forall k\ge 0.
\]
Iterating this inequality gives $V_\infty^\varepsilon(z_k)\le \beta_\varepsilon^{2k} V_\infty^\varepsilon(z_0)$. Now define $p_\varepsilon(x):=\sqrt{V_\infty^\varepsilon(x)}$.
By~\cref{lem:common-lyapunov-restricted-family}, $p_\varepsilon$ is also a norm, and there exists $C_\varepsilon>0$ such that
\[
\|x\|_2^2 \le V_\infty^\varepsilon(x)\le C_\varepsilon \|x\|_2^2,
\qquad \forall x\in\mathbb R^{|{\cal S}||{\cal A}|}.
\]
Thus, with $c_\varepsilon:=\sqrt{C_\varepsilon}$, we have
\[
\|x\|_2 \le p_\varepsilon(x)\le c_\varepsilon \|x\|_2,
\qquad \forall x\in\mathbb R^{|{\cal S}||{\cal A}|}.
\]
Therefore, we have
\[
\|z_k\|_2
\le
p_\varepsilon(z_k)
\le
\beta_\varepsilon^k p_\varepsilon(z_0)
\le
c_\varepsilon \beta_\varepsilon^k \|z_0\|_2.
\]
Using \cref{lem:projected-error-dynamics} once again,
\[
\dist_2(Q_k,\mathcal X_1)=\|z_k\|_2,
\qquad
\dist_2(Q_0,\mathcal X_1)=\|z_0\|_2,
\]
which proves the claim.
\end{proof}

\subsection{Statement and proof of Corollary~\ref{cor:fast-identification-pos}}
\label{app:proof-cor:fast-identification-pos}
\setcounter{corollary}{1}
\begin{corollary}[Fast finite-time identification of the POSS]\label{cor:fast-identification-pos}
Let~\cref{assump:optimal-class-separation} hold, fix any $\epsilon>0$ such that
\[
\beta_\epsilon:=\bar\rho+\epsilon \in (0,1),
\]
and let $c_\varepsilon> 0$ be the constant from~\cref{thm:global-exp-conv-X1}. Define
\[
K_{\mathrm{id}}
:=
\begin{cases}
0, & c_\varepsilon\,\dist_2(Q_0,\mathcal X_1)<\dfrac{\bar\Delta}{2},\\[1.2ex]
\left\lfloor
\dfrac{
\log\!\left(\dfrac{2c_\varepsilon\,\dist_2(Q_0,\mathcal X_1)}{\bar\Delta}\right)
}{
-\log\beta_\varepsilon
}
\right\rfloor+1,
& c_\varepsilon\,\dist_2(Q_0,\mathcal X_1)\ge\dfrac{\bar\Delta}{2}.
\end{cases}
\]
Then, we have $Q_k\in \mathcal X^*$ for all $k\ge K_{\mathrm{id}}$.
In particular,
\[
\pi_{Q_k}(s)\in \Phi^*(s),\qquad \forall s\in {\cal S},\ \forall k\ge K_{\mathrm{id}}.
\]
\end{corollary}
\begin{proof}
Write
\[
e_k=\alpha_k \mathbf{1}+z_k,\qquad z_k=\proj e_k.
\]
Since $\alpha_k \mathbf{1}$ does not change the action ordering, only $z_k$ matters for policy identification.
If $\|z_k\|_\infty<\frac{\bar\Delta}{2}$, then for every state $s\in {\cal S}_{\mathrm{sep}}$ and every action $b\notin\Phi^*(s)$,
\begin{align*}
\max_{a\in\Phi^*(s)} Q_k(s,a)-Q_k(s,b)
&=
V^*(s)-Q^*(s,b)
+
\max_{a\in\Phi^*(s)} e_k(s,a)-e_k(s,b)\\
&\ge
\bar\Delta_s
+
\max_{a\in\Phi^*(s)} z_k(s,a)-z_k(s,b)\\
&\ge
\bar\Delta_s-2\|z_k\|_\infty .
\end{align*}
Hence, it follows that
\[
\max_{a\in\Phi^*(s)}Q_k(s,a)-Q_k(s,b)
\ge
\bar\Delta_s-2\|z_k\|_\infty
\ge
\bar\Delta-2\|z_k\|_\infty
>0.
\]
Consequently, no non-optimal action can be tie-broken greedy, and therefore, we have $\pi_{Q_k}(s)\in\Phi^*(s)$ for all $s\in {\cal S}$.
By~\cref{eq:global-exp-conv-X1}, one gets
\[
\|z_k\|_2
=
\dist_2(Q_k,\mathcal X_1)
\le
c_\varepsilon\beta_\varepsilon^k\dist_2(Q_0,\mathcal X_1).
\]
The definition of $K_{\mathrm{id}}$ guarantees
\[
c_\varepsilon\beta_\varepsilon^k\dist_2(Q_0,\mathcal X_1)<\frac{\bar\Delta}{2},
\qquad \forall k\ge K_{\mathrm{id}}.
\]
Since $\|z_k\|_\infty\le \|z_k\|_2$, it follows that $Q_k\in \mathcal X^*$ for all $k\ge K_{\mathrm{id}}$.
\end{proof}

\subsection{Statement and proof of Theorem~\ref{thm:two-stage-convergence}}
\label{app:proof-thm:two-stage-convergence}
\setcounter{theorem}{1}
\begin{theorem}[Two-stage convergence]\label{thm:two-stage-convergence}
Let \cref{assump:optimal-class-separation} hold, and let $K_{\mathrm{id}}$ be defined as in \cref{cor:fast-identification-pos}. Define
\[
\bar{\cal H}_* := \{\bar A_\pi : \pi \in \Theta^*\},
\qquad
\bar\rho_* := \rho(\bar{\cal H}_*),
\]
where $\bar A_\pi := \proj(\gamma P\Pi^\pi)\proj$ for any $\pi \in \Theta^*$. Then, $\bar\rho_* \le \bar\rho \le \gamma$ and the following statements hold.
\begin{enumerate}
    \item For any $\varepsilon > 0$ such that $\beta_\varepsilon := \bar\rho+\varepsilon < 1$, there exists a constant $c_\varepsilon > 0$ such that
    \[
    \dist_2(Q_k,\mathcal X_1)
    \le
    c_\varepsilon \beta_\varepsilon^k \, \dist_2(Q_0,\mathcal X_1),
    \qquad \forall k \ge 0.
    \]
    In particular, the convergence toward the affine set \(\mathcal X_1\), and hence toward the POSS \(\mathcal X^*\), admits exponential upper bounds at any rate larger than the JSR $\bar\rho$ of the full restricted switching family, with the stochastic-policy convex hull handled by the Lyapunov inequality.

    \item For all $k \ge K_{\mathrm{id}}$, there exists a policy $\pi_k \in \Theta^*$ such that $Q_{k+1}-Q^* = A_{\pi_k}(Q_k-Q^*)$.

    \item For any $\varepsilon_* > 0$ such that $\beta_* := \bar\rho_*+\varepsilon_* < 1$, there exists a constant $\tilde C_{\varepsilon_*} > 0$ such that
    \[
    \|z_{K_{\mathrm{id}}+\ell}\|_2
    \le
    \tilde C_{\varepsilon_*}\beta_*^\ell
    \|z_{K_{\mathrm{id}}}\|_2,
    \qquad \forall \ell \ge 0.
    \]
    Thus, after finite-time identification of \(\mathcal X^*\), the transverse component admits exponential upper bounds at any rate larger than the JSR $\bar\rho_*$ of the restricted optimal family.

    \item If, in addition, $\bar\rho_* < \gamma$, then one may choose $\varepsilon_*>0$ sufficiently small so that $\beta_*=\bar\rho_*+\varepsilon_*<\gamma$. In this case, there exists a constant $D_{\varepsilon_*}>0$ such that
    \[
    \|Q_{K_{\mathrm{id}}+\ell}-Q^*\|_2
    \le
    D_{\varepsilon_*}\gamma^\ell
    \|Q_{K_{\mathrm{id}}}-Q^*\|_2,
    \qquad \forall \ell \ge 0.
    \]
\end{enumerate}
\end{theorem}
\begin{proof}
Write
\[
e_k := Q_k-Q^*,
\qquad
\alpha_k := \frac{1}{n}\mathbf{1}^\top e_k,
\qquad
z_k := \proj e_k,
\qquad
n := |{\cal S}||{\cal A}|.
\]
The global convergence estimate follows directly from \cref{thm:global-exp-conv-X1} applied to the full restricted switching family
$\bar A=\{\bar A_1,\ldots,\bar A_M\}$. Since $\mathcal X_1 \subset \mathcal X^*$ and \cref{cor:fast-identification-pos} gives finite-time entrance into \(\mathcal X^*\), this proves the first claim and shows that the convergence toward the POSS is controlled by any exponential rate larger than $\bar\rho$.
Next, by \cref{cor:fast-identification-pos}, for all $k\ge K_{\mathrm{id}}$, $Q_k \in \mathcal X^*$. Hence, for each such $k$, the tie-broken greedy policy $\pi_k := \pi_{Q_k}$ belongs to $\Theta^*$. Therefore, $Q_{k+1}=R+\gamma P\Pi^{\pi_k}Q_k$.
Since $\pi_k$ is optimal, it also satisfies $R+\gamma P\Pi^{\pi_k}Q^*=Q^*$.
Subtracting the two equations yields $Q_{k+1}-Q^*=A_{\pi_k}(Q_k-Q^*)$, where $A_{\pi_k}:=\gamma P\Pi^{\pi_k}$.
This proves the second claim.
Now define $e_k:=Q_k-Q^*$ and $z_k:=\proj e_k$. Then, for all $k\ge K_{\mathrm{id}}$,
\[
z_{k+1}
=
\proj e_{k+1}
=
\proj A_{\pi_k}e_k.
\]
Since
\[
e_k = \proj e_k + (I-\proj)e_k
\]
and $(I-\proj)e_k \in \mathrm{span}(\mathbf{1})$, there exists a scalar $c_k\in\mathbb{R}$ such that
\[
(I-\proj)e_k = c_k\mathbf{1}.
\]
Using $A_{\pi_k}\mathbf{1}=\gamma\mathbf{1}$ and $\proj\mathbf{1}=0$,
we obtain
\[
\proj A_{\pi_k}(I-\proj)e_k
=
c_k\proj A_{\pi_k}\mathbf{1}
=
c_k\gamma \proj\mathbf{1}
=
0.
\]
Therefore,
\[
z_{k+1}
=
\proj A_{\pi_k}\proj e_k
=
\bar A_{\pi_k}z_k,
\]
where $\bar A_{\pi_k}:=\proj A_{\pi_k}\proj$. Thus, after time \(K_{\mathrm{id}}\), the projected error evolves according to the switching family \(\bar{\cal H}_*=\{\bar A_\pi:\pi\in\Theta^*\}\).
Since \(\bar{\cal H}_*\subset\bar{\cal H}\), monotonicity of the JSR yields \(\bar\rho_* \le \bar\rho\).
On the other hand, \cref{lem:restricted-jsr-upper-gamma} gives \(\bar\rho \le \gamma\). Hence, \(\bar\rho_* \le \bar\rho \le \gamma\).
Fix any \(\varepsilon_*>0\) such that \(\beta_*:=\bar\rho_*+\varepsilon_*<1\).
Applying \cref{lem:common-lyapunov-restricted-family} to the restricted optimal family \(\bar{\cal H}_*\), there exists a common Lyapunov function for this family. Hence there exists a constant \(\tilde C_{\varepsilon_*}>0\) such that
\[
\|z_{K_{\mathrm{id}}+\ell}\|_2
\le
\tilde C_{\varepsilon_*}\beta_*^\ell
\|z_{K_{\mathrm{id}}}\|_2,
\qquad \forall \ell\ge 0.
\]
This proves the third claim.

It remains to control the component along $\mathbf{1}$. Write
\[
e_k=\alpha_k\mathbf{1}+z_k,
\qquad
\alpha_k:=\frac{1}{n}\mathbf{1}^\top e_k.
\]
Then
\[
\alpha_{k+1}
=
\frac{1}{n}\mathbf{1}^\top e_{k+1}
=
\frac{1}{n}\mathbf{1}^\top A_{\pi_k}e_k
=
\frac{1}{n}\mathbf{1}^\top A_{\pi_k}(\alpha_k\mathbf{1}+z_k).
\]
Since $A_{\pi_k}\mathbf{1}=\gamma\mathbf{1}$, we get
\begin{align*}
\alpha_{k+1}=\gamma\alpha_k+\eta_k,
\end{align*}
where $\eta_k:=\frac{1}{n}\mathbf{1}^\top A_{\pi_k}z_k$.
Define $L_*:=\max_{\pi\in\Theta^*}\frac{1}{n}\|\mathbf{1}^\top A_\pi\|_2$.
Then, $|\eta_k|\le L_*\|z_k\|_2$ so that
\begin{align*}
|{\alpha _{k + 1}}| \le \gamma |{\alpha _k}| + |{\eta _k}| \le \gamma |{\alpha _k}| + {L_*}{\left\| {{z_k}} \right\|_2}.
\end{align*}
Assume now that $\bar\rho_*<\gamma$.
Then we may choose $\varepsilon_*>0$ sufficiently small so that
\[
\beta_*=\bar\rho_*+\varepsilon_*<\gamma.
\]
Iterating the recursion for $\alpha_k$ from $K_{\mathrm{id}}$, we obtain
\[
|\alpha_{K_{\mathrm{id}}+\ell}|
\le
\gamma^\ell |\alpha_{K_{\mathrm{id}}}|
+
L_*\sum_{j=0}^{\ell-1}\gamma^{\ell-1-j}\|z_{K_{\mathrm{id}}+j}\|_2.
\]
Using the bound on $z_k$ gives
\[
|\alpha_{K_{\mathrm{id}}+\ell}|
\le
\gamma^\ell |\alpha_{K_{\mathrm{id}}}|
+
L_*\tilde C_{\varepsilon_*}
\sum_{j=0}^{\ell-1}\gamma^{\ell-1-j}\beta_*^j
\|z_{K_{\mathrm{id}}}\|_2.
\]
Since $\beta_*<\gamma$,
\[
\sum_{j=0}^{\ell-1}\gamma^{\ell-1-j}\beta_*^j
=
\frac{\gamma^\ell-\beta_*^\ell}{\gamma-\beta_*}
\le
\frac{\gamma^\ell}{\gamma-\beta_*}.
\]
Hence,
\[
|\alpha_{K_{\mathrm{id}}+\ell}|
\le
\gamma^\ell |\alpha_{K_{\mathrm{id}}}|
+
\frac{L_*\tilde C_{\varepsilon_*}}{\gamma-\beta_*}
\gamma^\ell \|z_{K_{\mathrm{id}}}\|_2.
\]
Finally, since $e_{K_{\mathrm{id}}+\ell}=\alpha_{K_{\mathrm{id}}+\ell}\mathbf 1+z_{K_{\mathrm{id}}+\ell}$, the triangle inequality yields
\begin{align*}
\|e_{K_{\mathrm{id}}+\ell}\|_2
&\le
\|\alpha_{K_{\mathrm{id}}+\ell}\mathbf 1\|_2+\|z_{K_{\mathrm{id}}+\ell}\|_2 \\
&=
\sqrt{n}\,|\alpha_{K_{\mathrm{id}}+\ell}|+\|z_{K_{\mathrm{id}}+\ell}\|_2 \\
&\le
\sqrt{n}\,\gamma^\ell
\left(
|\alpha_{K_{\mathrm{id}}}|
+
\frac{L_*\widetilde C_{\varepsilon_*}}{\gamma-\beta_*}\|z_{K_{\mathrm{id}}}\|_2
\right)
+
\widetilde C_{\varepsilon_*}\beta_*^\ell \|z_{K_{\mathrm{id}}}\|_2 \\
&\le
\gamma^\ell
\left[
\sqrt{n}\,|\alpha_{K_{\mathrm{id}}}|
+
\left(
\frac{\sqrt{n}\,L_*\widetilde C_{\varepsilon_*}}{\gamma-\beta_*}
+
\widetilde C_{\varepsilon_*}
\right)\|z_{K_{\mathrm{id}}}\|_2
\right],
\end{align*}
where in the last step we used \(\beta_*^\ell\le \gamma^\ell\).
Finally, we bound \(|\alpha_{K_{\mathrm{id}}}|\) and \(\|z_{K_{\mathrm{id}}}\|_2\) by \(\|e_{K_{\mathrm{id}}}\|_2\).
By the definition of \(\alpha_{K_{\mathrm{id}}}\),
\[
\sqrt{n}\,|\alpha_{K_{\mathrm{id}}}|
=
\sqrt{n}\left|\frac{1}{n}\mathbf 1^\top e_{K_{\mathrm{id}}}\right|
=
\frac{1}{\sqrt{n}}|\mathbf 1^\top e_{K_{\mathrm{id}}}|
\le
\frac{1}{\sqrt{n}}\|\mathbf 1\|_2\,\|e_{K_{\mathrm{id}}}\|_2
=
\|e_{K_{\mathrm{id}}}\|_2.
\]
Moreover, since \(z_{K_{\mathrm{id}}}=\proj e_{K_{\mathrm{id}}}\) and \(\proj\) is an orthogonal projection, $\|z_{K_{\mathrm{id}}}\|_2\le \|e_{K_{\mathrm{id}}}\|_2$. Substituting these two estimates into the previous bound, we conclude that
\[
\|e_{{K_{\mathrm{id}}}+\ell}\|_2
\le
D_{\varepsilon_*}\gamma^\ell \|e_{K_{\mathrm{id}}}\|_2,
\qquad \forall \ell\ge 0,
\]
where one may take, for instance,
\[
D_{\varepsilon_*}
:=
1+\widetilde C_{\varepsilon_*}
+\frac{\sqrt{n}\,L_*\widetilde C_{\varepsilon_*}}{\gamma-\beta_*}.
\]
Recalling that \(e_k=Q_k-Q^*\), we finally obtain
\[
\|Q_{K_{\mathrm{id}}+\ell}-Q^*\|_2
=
\|e_{K_{\mathrm{id}}+\ell}\|_2
\le
D_{\varepsilon_*}\gamma^\ell
\|Q_{K_{\mathrm{id}}}-Q^*\|_2,
\qquad \forall \ell\ge 0.
\]
This completes the proof.
\end{proof}

\subsection{Statement and proof of Corollary~\ref{cor:unique-optimal-policy}}
\label{app:proof-cor:unique-optimal-policy}
\setcounter{corollary}{2}
\begin{corollary}[Unique optimal policy case]\label{cor:unique-optimal-policy}
Suppose that the optimal policy is unique, namely, $\Theta^*=\{\pi^*\}$.
Under the assumptions of \cref{thm:two-stage-convergence}, we then have
\[
\pi_{Q_k}=\pi^*, \qquad \forall k\ge K_{\mathrm{id}}.
\]
Consequently, for all $k\ge K_{\mathrm{id}}$, $Q_{k+1}-Q^*=A_{\pi^*}(Q_k-Q^*)$, and $z_{k+1}=\bar A_{\pi^*}z_k$, where $\bar A_{\pi^*}:=\proj(\gamma P\Pi^{\pi^*})\proj$.
Moreover,
\[
\rho(\bar A_{\pi^*})=\gamma |\lambda_2(P\Pi^{\pi^*})|.
\]
Hence, in the unique-optimal-policy case, the post-identification transverse dynamics are governed by a single linear system rather than a switching family.
\end{corollary}
\begin{proof}
Since the optimal policy is unique, for each state $s\in {\cal S}$ the set $\Phi^*(s)$ is a singleton, namely, $\Phi^*(s)=\{\pi^*(s)\}$. By \cref{cor:fast-identification-pos}, we have $Q_k\in \mathcal X^*$ for all $k\ge K_{\mathrm{id}}$. Therefore, $\pi_{Q_k}(s)\in \Phi^*(s)=\{\pi^*(s)\}$ for all $s\in {\cal S}$, which implies
$\pi_{Q_k}=\pi^*$ for all $k\ge K_{\mathrm{id}}$.
The stated dynamics then follow immediately from \cref{thm:two-stage-convergence}, and the identity $\rho(\bar A_{\pi^*})=\gamma |\lambda_2(P\Pi_{\pi^*})|$ follows from \cref{lem:restricted-spectral-radius}.
\end{proof}

\subsection{Proof of Lemma~\ref{lem:dobrushin-Diameter-properties}}
\label{app:proof-lem:dobrushin-Diameter-properties}
\begin{proof}
The bounds \(0\le \tau_{\rm D}(B)\le 1\), the equivalence between
\(\tau_{\rm D}(B)<1\) and the scrambling property, the submultiplicativity
\(\tau_{\rm D}(BC)\le \tau_{\rm D}(B)\tau_{\rm D}(C)\), and the exact
diameter-seminorm characterization
\[
\tau_{\rm D}(B)
=
\sup_{v\notin\operatorname{span}(\mathbf 1)}
\frac{\|Bv\|_{\rm dm}}{\|v\|_{\rm dm}}
\]
are standard properties of coefficients of ergodicity for finite stochastic
matrices~\cite{dobrushin1956central,hajnal1958weak,seneta2006non}. The convexity
of \(\tau_{\rm D}\) follows directly from its definition as a maximum of
total-variation distances between rows. Finally, the diameter seminorm is invariant
under additive shifts by constants, which gives
\(\|v+c\mathbf 1\|_{\rm dm}=\|v\|_{\rm dm}\). Since \(B\) is row-stochastic,
\(B\mathbf 1=\mathbf 1\). Therefore the projection only removes or adds
components in \(\operatorname{span}(\mathbf 1)\), and hence
\[
\|\proj B\proj v\|_{\rm dm}
=
\|B\proj v\|_{\rm dm}.
\]
Taking the supremum over nonconstant \(v\) gives the last equality. This
completes the proof.
\end{proof}

\subsection{Proof of Proposition~\ref{prop:uniform-scrambling-strict}}
\label{app:proof-prop:uniform-scrambling-strict}
\begin{proof}
Let
\[
\bar{\mathcal B}:=\{\bar B_\pi:\pi\in\Theta\},
\qquad
\bar B_\pi:=\proj B_\pi\proj .
\]
Since \(\bar A_\pi=\gamma\bar B_\pi\), we have
\[
\bar\rho=\gamma\rho(\bar{\mathcal B}).
\]

The only subtle point is that \(\|\cdot\|_{\rm dm}\) is not a norm on the full space \(\mathbb R^{|\mathcal Y|}\), but it becomes a genuine norm once the constant direction is removed. Let
\[
W:=\operatorname{span}(\mathbf 1)^\perp,
\qquad
\|w\|_{\rm dm}:=\max_{x\in \mathcal Y}w_x-\min_{x\in \mathcal Y}w_x,
\quad w\in W.
\]
On \(W\), \(\|\cdot\|_{\rm dm}\) is a norm. For each \(\pi\in\Theta\), define
\[
T_\pi:=\bar B_\pi|_W:W\to W.
\]
We identify \(\bar B_\pi\) with its restriction \(T_\pi\) on \(W\). Since \(\bar B_\pi\mathbf 1=0\) and \(\bar B_\pi(\mathbb R^{|\mathcal Y|})\subset W\), each \(\bar B_\pi\) has the block form
\[
\bar B_\pi \sim
\begin{bmatrix}
0 & 0\\
0 & T_\pi
\end{bmatrix}
\qquad
\text{with respect to }
\mathbb R^{|\mathcal Y|}=\operatorname{span}(\mathbf 1)\oplus W.
\]
Thus, the JSR of the full matrices \(\bar B_\pi\) equals the JSR of the restricted operators \(T_\pi\) on \(W\).

Now fix \(\omega=(\pi_0,\ldots,\pi_{k-1})\in\Theta^k\). By the same projection-product identity as in~\cref{lem:projection-product-identity},
\[
T_{\pi_{k-1}}\cdots T_{\pi_0}
=
(\proj B_{\pi_{k-1}}\cdots B_{\pi_0}\proj)\big|_W
=
(\proj B_\omega\proj)\big|_W .
\]
By~\cref{lem:dobrushin-Diameter-properties}(5)--(6), the induced operator norm of this map on \((W,\|\cdot\|_{\rm dm})\) is exactly \(\tau_{\rm D}(B_\omega)\). Therefore,
\[
\|T_{\pi_{k-1}}\cdots T_{\pi_0}\|_{{\rm dm}\to{\rm dm}}
=
\tau_{\rm D}(B_\omega),
\]
and hence
\[
\rho(\bar{\mathcal B})
=
\lim_{k\to\infty}
\max_{\omega\in\Theta^k}
\tau_{\rm D}(B_\omega)^{1/k}.
\]

We first prove the equivalence between items 1) and 3). Suppose that \(\bar\rho<\gamma\). Then
\[
\rho(\bar{\mathcal B})<1.
\]
Choose \(\beta\) such that \(\rho(\bar{\mathcal B})<\beta<1\). By the definition of the JSR on the normed space \((W,\|\cdot\|_{\rm dm})\), there exists an integer \(L\ge 1\) such that
\[
\max_{\omega\in\Theta^L}
\|T_\omega\|_{{\rm dm}\to{\rm dm}}^{1/L}
\le
\beta .
\]
Using \(\|T_\omega\|_{{\rm dm}\to{\rm dm}}=\tau_{\rm D}(B_\omega)\), we obtain
\[
\max_{\omega\in\Theta^L}
\tau_{\rm D}(B_\omega)
\le
\beta^L
<1,
\]
which proves item 3).

Conversely, suppose that item 3) holds for some \(L\ge 1\), and define
\[
\theta_L:=\max_{\omega\in\Theta^L}\tau_{\rm D}(B_\omega)<1.
\]
Then every length-\(L\) product of the family \(\{T_\pi:\pi\in\Theta\}\) has operator norm at most \(\theta_L\). By submultiplicativity, every length-\(mL\) product has operator norm at most \(\theta_L^m\). Hence
\[
\rho(\bar{\mathcal B})
\le
\theta_L^{1/L}
<
1.
\]
Multiplying by \(\gamma\), we obtain \(\bar\rho<\gamma\). This proves item 1).

The equivalence between items 2) and 3) follows from~\cref{lem:dobrushin-Diameter-properties}: for each row-stochastic matrix \(B\), \(\tau_{\rm D}(B)<1\) if and only if \(B\) is scrambling. Since \(\Theta^L\) is finite, all length-\(L\) deterministic products are scrambling if and only if
\[
\max_{\omega\in\Theta^L}\tau_{\rm D}(B_\omega)<1.
\]

We next relate the deterministic and stochastic formulations. For stochastic policies \(\mu_0,\ldots,\mu_{L-1}\), each \(B_{\mu_j}\) is a convex combination of deterministic-policy matrices:
\[
B_{\mu_j}
=
\sum_{\pi_j\in\Theta}c_{\pi_j}(\mu_j)B_{\pi_j}.
\]
Therefore,
\[
B_{\mu_{L-1}}\cdots B_{\mu_0}
=
\sum_{\omega\in\Theta^L}
C_\omega B_\omega,
\]
where \(C_\omega\ge 0\) and \(\sum_{\omega\in\Theta^L}C_\omega=1\).
By the convexity of \(\tau_{\rm D}\),
\[
\tau_{\rm D}\!\left(
B_{\mu_{L-1}}\cdots B_{\mu_0}
\right)
\le
\sum_{\omega\in\Theta^L}C_\omega\tau_{\rm D}(B_\omega)
\le
\max_{\omega\in\Theta^L}\tau_{\rm D}(B_\omega).
\]
Thus item 3) implies item 4). The reverse implication is immediate because deterministic policies are special cases of stochastic policies. Hence items 3) and 4) are equivalent.

Finally, if the equivalent conditions hold, then by item 2) every \(B_\omega\), \(\omega\in\Theta^L\), is scrambling. Since the family \(\{B_\omega:\omega\in\Theta^L\}\) is finite, the quantity
\[
\eta_L:=\min_{\omega\in\Theta^L}\min_{x,y\in \mathcal Y}
\sum_{\upsilon\in \mathcal Y}\min\{(B_\omega)_{x\upsilon},(B_\omega)_{y\upsilon}\}
\]
is strictly positive. By~\cref{def:dobrushin-coefficient},
\[
\tau_{\rm D}(B_\omega)\le 1-\eta_L,
\qquad
\forall \omega\in\Theta^L.
\]
Hence every length-\(L\) product has operator norm at most \(1-\eta_L\), and
\[
\rho(\bar{\mathcal B})
\le
(1-\eta_L)^{1/L}.
\]
Therefore,
\[
\bar\rho
=
\gamma\rho(\bar{\mathcal B})
\le
\gamma(1-\eta_L)^{1/L}
<
\gamma.
\]
This completes the proof.
\end{proof}

\end{document}